\documentclass[12pt]{amsart}

\usepackage{hyperref}
\usepackage{multicol}
\usepackage{amsmath,amscd,amssymb,amsthm,pinlabel}
\usepackage{array}
\usepackage{graphicx}
\usepackage[usenames,dvipsnames,svgnames,table]{xcolor}
\usepackage{comment}

\usepackage{fancyhdr}

\newtheorem{theorem}{Theorem}[section]
\newtheorem{lemma}[theorem]{Lemma}
\newtheorem{proposition}[theorem]{Proposition}
\newtheorem{corollary}[theorem]{Corollary}

\newtheorem{definition}{Definition}
\theoremstyle{definition}

\newtheorem{question}[theorem]{Question}

\newtheorem{example}{Example}

\newcommand{\Mod}{\mathrm{Mod}}

\newcommand{\Out}{\mathrm{Out}}

\newcommand{\C}{\mathcal{C}}
\newcommand{\T}{\mathcal{T}}
\newcommand{\V}{\mathcal{C}^0}

\newcommand{\supp}{\mathrm{supp}}
\newcommand{\grp}[1]{\langle{#1}\rangle}

\newcommand{\F}{\mathbb{F}}
\newcommand{\Z}{\mathbb{Z}}
\newcommand{\R}{\mathbb{R}}
\renewcommand{\i}{\infty}
\newcommand{\s}{\sigma}

\newcommand{\del}{\partial}

\newcommand{\tilda}[1]{\widetilde{#1}}
\newcommand{\subs}{\textsc{subs}}

\newcommand{\syl}{\mathrm{syl}}

\newcommand{\Min}{\mathrm{Min}}
\newcommand{\Fill}{\textsf{span}}
\newcommand{\lk}{\mathrm{lk}}

\begin{document}

\title[Convex cocompactness in $\Mod(S)$ via quasiconvexity in RAAGs]{Convex cocompactness in mapping class groups via\\quasiconvexity in right-angled Artin groups}
\author[J. Mangahas]{Johanna Mangahas}
\address{Mathematics Department,
Brown University,
151 Thayer Street,
Providence, RI 02912, U.S.A.}
\email{\href{mailto:mangahas@math.brown.edu}{mangahas@math.brown.edu}}
\author[S. Taylor]{Samuel J. Taylor}
\address{Department of Mathematics, 
University of Texas at Austin, 
1 University Station C1200, 
Austin, TX 78712, U.S.A.
}
\email{\href{mailto:staylor@math.utexas.edu}{staylor@math.utexas.edu}}
\thanks{\tiny The first author acknowledges support by NSF DMS-1204592 and the second author is partially supported by NSF RTG grants DMS-0636557 and DMS-1148490. }
\date{June 2013}
\begin{abstract}
We characterize convex cocompact subgroups of mapping class groups that arise as subgroups of specially embedded right-angled Artin groups. That is, if the right-angled Artin group $G < \Mod(S)$ satisfies certain conditions that imply $G$ is quasi-isometrically embedded in $\Mod(S)$, then a purely pseudo-Anosov $H<G$ is convex cocompact in $\Mod(S)$ if and only if it is combinatorially quasiconvex in $G$. We use this criterion to construct convex cocompact subgroups of $\Mod(S)$ whose orbit maps into the curve complex have small Lipschitz constants.

\end{abstract}
\maketitle


\section{Introduction}\label{intro}

\emph{Convex cocompact} subgroups of mapping class groups are those finitely generated subgroups whose orbits in Teichm\"{u}ller space are quasiconvex \cite{FMo}, or equivalently, whose orbit map into the curve complex defines a quasi-isometric embedding of the group \cite{H,KL2}. Such subgroups of mapping class groups are of interest because of their close connection to surface group extensions.  Letting $S$ denote the close surface of genus $g \ge2$ and $\mathring{S}$ denote $S$ punctured at $p \in S$, the well-known Birman exact sequence 
 $$1 \longrightarrow \pi_1(S,p) \longrightarrow \Mod(\mathring{S}) \overset{f}{\longrightarrow} \Mod(S )\to 1$$
gives rise to an extension $E_G$ of $\pi_1(S,p)$ for each subgroup $G < \Mod(S)$, obtained by setting $E_G = f^{-1}(G)$.  Theorems in \cite{FMo} and \cite{H} combine to say that $E_G$ is Gromov-hyperbolic if and only if $G$ is convex cocompact. Hence, to understand the prevalence and properties of hyperbolic surface group extensions we are left to study convex cocompact subgroups of $\Mod(S)$.

The results of this paper determine the conditions of convex cocompactness for mapping class subgroups contained in certain \emph{admissible} embedded right-angled Artin groups, including the groups constructed in \cite{CLM}.  We denote by $A(\Gamma)$ the right-angled Artin group associated to the graph $\Gamma$, and say a subgroup $H < A(\Gamma)$ is \emph{quasiconvex in} $A(\Gamma)$ if it is quasiconvex in the word metric using the generating set corresponding to vertices of $\Gamma$.

\begin{theorem}\label{intromain}Suppose $A(\Gamma) < \Mod(S)$ is admissible.  Then $H < A(\Gamma)$ is convex cocompact if and only if it is quasiconvex in $A(\Gamma)$ and all nontrivial elements of $H$ are pseudo-Anosov.
\end{theorem}

As observed in \cite{KLSurvey}, the existence of a purely pseudo-Anosov subgroup $H \le \Mod(S)$ that is \emph{not} convex cocompact would imply that the extension $E_H$ has a finite $K(E_H,1)$ and no Baumslag-Solitar subgroups but is \emph{not} hyperbolic. A well-know conjecture attributed to Gromov asserts that such a group does not exist (see \cite{KLSurvey} and \cite{FMo}). By Theorem \ref{intromain} an affirmative answer to the following questions would produce a counter-example to this conjecture.

\begin{question}Does there exist an admissible embedding $A(\Gamma)\hookrightarrow \Mod(S)$ for which some non-quasiconvex $H < A(\Gamma)$ is all-pseudo-Anosov?\end{question}

The condition for convex cocompactness in Theorem \ref{intromain} comes as corollary to a stronger, constructive result. In this form, we are able to identify when $H < A(\Gamma)$ is purely pseudo-Anosov by checking only finitely many mapping classes. Below, $f(\Gamma,K)$ is a positive integer valued function depending only on $\Gamma$ and a $K \ge 0$, and $A(\Gamma)$-\emph{length} is word length in $A(\Gamma)$ with respect to the generating set corresponding to vertices of $\Gamma$.

\begin{theorem}\label{construction}Suppose $A(\Gamma) < \Mod(S)$ is admissible and $H$ is $K$-quasiconvex in $\Gamma$.  Let $L = f(\Gamma,K)$.  Then $H$ is generated by words of $A(\Gamma)$-length less than $L$; if these are pseudo-Anosov, then $H$ is convex cocompact.\end{theorem}

In particular, Theorem \ref{construction} provides flexible means to build explicit examples of convex cocompact subgroups of mapping class groups, distinguishing it from other constructions in the literature. See Section \ref{sec:examples} for some examples.

We briefly survey other methods of producing convex cocompact subgroups of mapping class groups.  The simplest are free groups generated by sufficiently high powers of any finite family of independent pseudo-Anosov mapping classes \cite{FMo}.  Later, Min \cite{Min} created virtually free examples isomorphic to $G*H$ for arbitrary finite subgroups $G,H < \Mod(S)$, by conjugating one of these groups by a sufficiently high power of a pseudo-Anosov.  A third set of examples live in certain hyperbolic groups embedded in the mapping class group $\Mod(\mathring{S})$ of a once-punctured surface $\mathring{S}$; these are the surface group extensions $E_G$ described above.  Generalizing \cite{KLS}, Dowdall, Kent, and Leininger prove that, when $E_G$ is hyperbolic, its quasiconvex all-pseudo-Anosov subgroups are convex cocompact in $\Mod(\mathring{S})$ \cite{DKL}.  The convex cocompact subgroups considered in this paper are most similar in spirit to these last examples, with $E_G$ replaced by $A(\Gamma)$ and cubical CAT(0) geometry playing a role similar to hyperbolicity. Here, the idea is to replace quasi-convex orbits in Teichm\"{u}ller space with combinatorially quasiconvex orbits in $\tilda{S}_{\Gamma}$, the CAT(0) cube complex associated to $A(\Gamma).$

To highlight a difference between the all-pseudo-Anosov free groups constructed by our method and earlier examples, we describe a family of convex cocompact subgroups of $\Mod(S_g)$ whose orbit maps into the curve complex have Lipschitz constants on the order of $1/g$. Let $\C(S)$ denote the curve complex of $S$ and $\ell_S(f)$ the stable translation length of $f \in \Mod(S)$ in $\C(S)$ (see Section  \ref{basics} for definitions).

\begin{theorem}\label{smalltrans}
Let $S_g$ be a surface of genus $g$ for some $g \ge3$. Then for any $N\ge1$ there exists a convex cocompact $H = \langle w_1, \ldots w_N \rangle < \Mod(S_g)$ with the following property: there is an $\alpha \in \mathcal{C}^0(S_g)$ so that for any $h \in H$,
\[d_{\mathcal{C}(S_g)}(\alpha,h\alpha) \leq |h|_H\cdot\frac{4}{g-1} + 2 ,\]
where $|.|_H$ denotes word length in $H$ with the given generators. In particular, $\ell_S(w_i) \le \frac{4}{g-1} $ for each $i = 1, \ldots, N$.

\end{theorem}

The pseudo-Anosovs $w_i$ appearing in Theorem \ref{smalltrans} are themselves $g$th powers. Since it is known that the minimum stable translation length in the curve complex is of order $1/g^2$ \cite{GT}, it would be interesting to construct rank $\ge 2$ convex cocompact subgroups whose orbit maps into the curve complex have Lipschitz constants on the order of $1/g^2$.

This paper is organized as follows: Section $2$ contains background on surfaces and mapping class groups. Section $3$ briefly reviews some cubical geometry and proves the main technical property of quasiconvex subgroups of right-angled Artin groups used in this paper. The definition of admissibility is then given in Section $4$, where some properties of admissible homomorphisms from right-angled Artin groups into mapping class groups are established. The proof of Theorem \ref{construction} is given in Section $5$. Section $6$ completes the proof of Theorem \ref{intromain} and Section $7$ contains explicit constructions of convex cocompact mapping class subgroups. The paper concludes with Section $8$, which proves Theorem \ref{smalltrans}.

\textbf{Acknowledgments.}  We thank Centre de Recerca Matem\`{a}tica and the Polish Academy of Sciences for hosting both authors during parts of this research.  

\section{Surfaces and mapping classes} \label{basics}
\subsection{Quasiconvexity and quasi-isometry}\label{quasi}
Let $(X,d_X)$ and $(Y,d_Y)$ be metric spaces.  For constants $K\ge 1$ and $L \ge 0$, a map $f: X \to Y$ is a \emph{$(K,L)$-quasi-isometric embedding} if for all $x_1,x_2 \in X$
$$\frac{d_X(x_1,x_2) -L}{K} \le d_Y(f(x_1),f(x_2)) \le K (d_X(x_1,x_2) +L). $$
In addition, if every point in $Y$ is within a bounded distance from the image $f(X)$, then $f$ is a \emph{quasi-isometry} and $X$ and $Y$ are said to be \emph{quasi-isometric}.  Where $I$ is a subinterval of $\R$ or $\Z$, we call a $(K,L)$-quasi-isometric embedding $f:I \to Y$ a \emph{$(K,L)$-quasi-geodesic}. If $K =1$ and $L=0$, then $f: I \to Y$ is a \emph{geodesic}.  We say $Y$ is a \emph{geodesic metric space} if for all $y_1,y_2 \in Y$, there is a a geodesic $f:[a,b] \to Y$ with $f(a) =y_1$ and $f(b)=y_2$. For example, giving unit length to each edge of a graph $G$, i.e. a $1$-dimensional CW complex, turns $G$ into a geodesic metric space by taking the induced path metric.

For convenience, we write $K$-quasigeodesic or $K$-quasi-isometric embedding to mean a $(K,K)$-quasigeodesic or $(K,K)$-quasi-isometric embedding respectively.

We say a subset $X'$ of a geodesic metric space $X$ is \emph{$K$-quasiconvex} if for any $x,y \in X'$ and any geodesic $[x,y]$ between $x,y$ in $X$,
\[[x,y] \subset N_{K}(X').\]  We say $X'$ is \emph{quasiconvex} if it is $K$-quasiconvex for some $K$.  When we speak of a quasiconvex subgroup $H$ of a group $G$, we have fixed a word metric on $G$ with respect to some finite generating set (changing generating sets can change which subgroups are quasiconvex).  For CAT(0) cube complexes we additionally define \emph{combinatorial quasiconvexity} in Section \ref{cubegeo}.

\subsection{Surface topology basics}
Here we recall some relevant information about surfaces. For additional details, we refer the reader to \cite{MM2, FM}. Fix a surface $S$ of genus $g$ with $p$ punctures and $b$ boundary components. The \emph{complexity} of $S$ is the quantity $\xi(S) =3g-3+p+b$. In this paper we will only consider surfaces with $\xi(S) \ge 1$ and we will often not distinguish between boundary components and punctures of $S$. By a \emph{subsurface} $X$ of $ S$, we mean a compact submanifold such that the homomorphism $\pi_1X \to \pi_1S$ induced by inclusion is injective. Hence, all subsurfaces are assumed to be \emph{essential}. The subsurface $X \subset S$ is \emph{nonannular} if $X$ is not homeomorphic to an annulus.  Denote by $\Omega(S)$ the set of isotopy classes of nonannular subsurfaces of $S$. Although annuli play an important part in the analysis of \cite{MM2} and subsequence work, they are not considered in this paper. A \emph{curve} $\gamma$ in $S$ is an essential, simple loop in $S$, i.e. the image of a $\pi_1$-injective embedding of the circle into $S$. Recall that $\gamma$ is \emph{essential} if it is homtopically nontrivial and not parallel to a boundary component or puncture. As is standard in the subject, we often do not distinguish between a curve and its isotopy class or a subsurface and its isotopy class.

The \emph{curve complex} $\C(S)$ of the surface $S$ is the graph with vertex set the collection of isotopy classes of curves in $S$. Vertices $v$ and $w$ are joined by an edge in $\C(S)$ if $v$ and $w$ have disjoint representatives in $S$. When $S$ has complexity $1$, that is when $S$ is a once-punctured torus or a four-times punctured sphere, this definition produces a graph without edges and so is modified as follows: the vertices of $\C(S)$ are unchanged, but two vertices are joined by an edge if the corresponding curves have the minimal number of intersections among pairs of curves on $S$. Hence, curves on the once-punctured surface intersecting once are joined by an edge, as are curves intersecting twice on the four-times punctured sphere. With this definition, if $S$ has complexity $1$ then $\C(S)$ is the standard Farey graph. In general, we consider $\C(S)$ with its standard graph metric where each edge is assigned unit length.  In Section \ref{mcgbasics} we describe the action of $\Mod(S)$ on $\C(S)$ by isometry.  The foundational result in the study of the coarse geometry of the mapping class group is the following:

\begin{theorem}\cite{MM1}
For $S$ with $\xi(S) \ge 1$, $\C(S)$ is Gromov hyperbolic.
\end{theorem}

Recall that a geodesic metric space is Gromov hyperbolic if there exists a $\delta \ge 0$ so that for any points $x,y,z$ and geodesics $[x,y], [y,z]$ and $[z,x]$ between the three points,
$$[x,y] \subset N_{\delta}([y,z] \cup [z,x]) $$ 
where $N_{\delta}$ denotes a $\delta$-neighborhood in $\C(S)$.

A \emph{pants decomposition} $P$ of $S$ is a maximal collection of pairwise connected vertices of $\C(S)$, or in terms of the surface, a maximal collection of isotopy classes of pairwise disjoint curves in $S$. A \emph{marking} $\mu$ of $S$ is a pants decomposition $P = \{ \gamma_1, \ldots, \gamma_{\xi(S)} \}$ with the following additional structure: for each $\gamma_i \in P$ there is a corresponding curve $\beta_i$ contained in $S \setminus (P \setminus \gamma_i)$ that intersects $\gamma_i$ in the minimal possible number of times. In other words, if $X$ is the complexity $1$ component of $S \setminus (P \setminus \gamma_i)$ then $\beta_i$ is any adjacent curve to $\gamma_i$ in $\C(X)$. Then $\mu = \{(\gamma_1,\beta_1), \ldots, (\gamma_{\xi(S)}, \beta_{\xi(S)})  \}$ is a marking of $S$. The underlying pants decomposition of the marking $\mu$ is called the \emph{base} of $\mu$ and is denoted base($\mu$). In the terminology of \cite{MM2}, what we have described is a complete clean marking of $S$ and in this reference the authors construct the marking complex $\mathcal{M}(S)$ whose vertices are markings of $S$ and edges are determined by certain elementary moves on markings. The details will not be reviewed here; however, it suffices for us to recall that $\mathcal{M}(S)$ is a locally finite, connected graph and the natural action of $\Mod(S)$ is proper and cocompact. Hence, by the Svarc-Milnor lemma, $\mathcal{M}(S)$ is quasi-isometric to $\Mod(S)$. Again, see \cite{MM2} for details.

For a curve $\gamma$ and a nonannular subsurface $X$ of $S$ we define the \emph{subsurface projection of $\gamma$ to $X$}, denoted $\pi_X(\gamma)$, as follows: first realize $\gamma$ and $\partial X$ with minimal intersection, for example by taking geodesic representatives in some hyperbolic metric on $S$. If $\gamma$ does not intersect $\partial X$ then either $\gamma \subset X$, in which case we set $\pi_X(\gamma) = \{\gamma\}$, or $\gamma$ does not intersect $X$ and $\pi_X(\gamma)$ is the empty set. Otherwise, $\gamma \cap X =\{\gamma_1,\ldots, \gamma_k \}$ is an nonempty collection of essential arcs in $X$ and $\pi_X(\gamma)$ is the subset of $\C(X)$ whose elements are isotopic to the boundary of a regular neighborhood of the union of $\partial X$ and $\gamma_i$ for some $1\le i \le k$. In other words, the curves of $\pi_X(\gamma)$ arise from performing surgery along $\partial X$ to the arcs of $\gamma \cap \partial X$. This gives a subset of $\C(X)$ with diam$_X(\pi_X(\gamma)) \le 3$, where diam$_X$ denotes the diameter of a collection of vertices in $\C(X)$. Observe that if $\pi_X(\gamma) = \emptyset$ then $\gamma \in N_1(\partial X)$. When $\pi_X(\gamma) \neq \emptyset $, then we say that $\gamma$ \emph{cuts} $X$; otherwise $\gamma$ \emph{misses} $X$.

The following result of Masur-Minsky gives control over the projection of a geodesic in the curve complex to the curve complex of a subsurface. Its main application in this paper is to force curve complex geodesics to run though prescribed regions of $\C(S)$ thereby guaranteeing a definite length.

\begin{theorem}[Bounded Geodesic Image Theorem \cite{MM2}]\label{BGIT}
Given $S$ as above, there is $K_{BGI} \ge 0$ so that if $g$ is a geodesic in $\C(S)$ and $Y$ is a subsurface of $S$ with $\mathrm{diam}_{Y}(g) \ge K_{BGI}$ then there is a vertex $\alpha$ of $g$ with $\pi_X(\alpha) = \emptyset$.
\end{theorem}

We can also project markings to the curve complex of a subsurface. For a marking $\mu$ and (nonannular) subsurface $X$ the projection of $\mu$ to $\C(X)$ is defined as
$$\pi_X(\mu) = \cup_{\gamma \in \mathrm{base}(\mu)} \pi_X(\gamma) .$$
For $\alpha, \beta$ either curves or markings, we set
$$d_X(\alpha,\beta) = \mathrm{diam}_{\C(X)}(\pi_X(\alpha) \cup \pi_X(\beta)),$$
when defined. With this notation, it is well known that the projection from the marking complex to the curve complex of a subsurface is coarsely $4$-Lipschitz. 

Given connected, non-isotopic, proper subsurfaces $X$ and $Y$ of $S$, there are three possibilities for their relative position in $S$ and these possibilities are captured by subsurface projections. If $X$ and $Y$ are \emph{disjoint}, then $\pi_X(\partial Y)= \emptyset = \pi_Y(\partial X)$. $X$ and $Y$ are \emph{nested} if, up to switching $X$ and $Y$, $X\subset Y$ up to isotopy in which case $\pi_Y(\partial X) \neq \emptyset$ but $\pi_X(\partial Y) = \emptyset$. Finally, if $X$ and $Y$ are neither disjoint nor nested, then $X$ and $Y$ \emph{overlap} and $\pi_X(\partial Y) \neq \emptyset \neq \pi_Y(\partial X)$. We use the notation $X \pitchfork Y$ to denote that $X$ and $Y$ overlap. A collection $\mathbb{X}=\{X_1, \ldots, X_n\}$ of nonannular proper subsurfaces of $S$ is \emph{cleanly embedded} if whenever two subsurface in the collection intersect essentially they overlap, i.e. 
$$X_i \cap X_j \neq \emptyset \implies X_i \pitchfork X_j  \quad \mathrm{for} \  i \neq j .$$
In this case, we can consider the associated \emph{coincidence graph} $\Gamma_{\mathbb{X}}$ whose vertices are labeled $v_1, \ldots, v_n$ in correspondence to the subsurfaces of $\mathbb{X}$ and $v_i$ is joined by an edge to $v_j$ if and only if $X_i \cap X_j = \emptyset$. See, for example, Figure \ref{fig1}. In general, a collection of subsurfaces $\{X_1, \ldots, X_n\}$ \emph{fills} the surface $S$ if for any curve $\gamma$, there is an $i$ such that $\pi_{X_i}(\gamma) \neq \emptyset$. That is, there is a subsurface $X_i$ in the collection that is cut by $\gamma$.

\begin{figure}[htbp]
\begin{center}
\includegraphics[height=50mm]{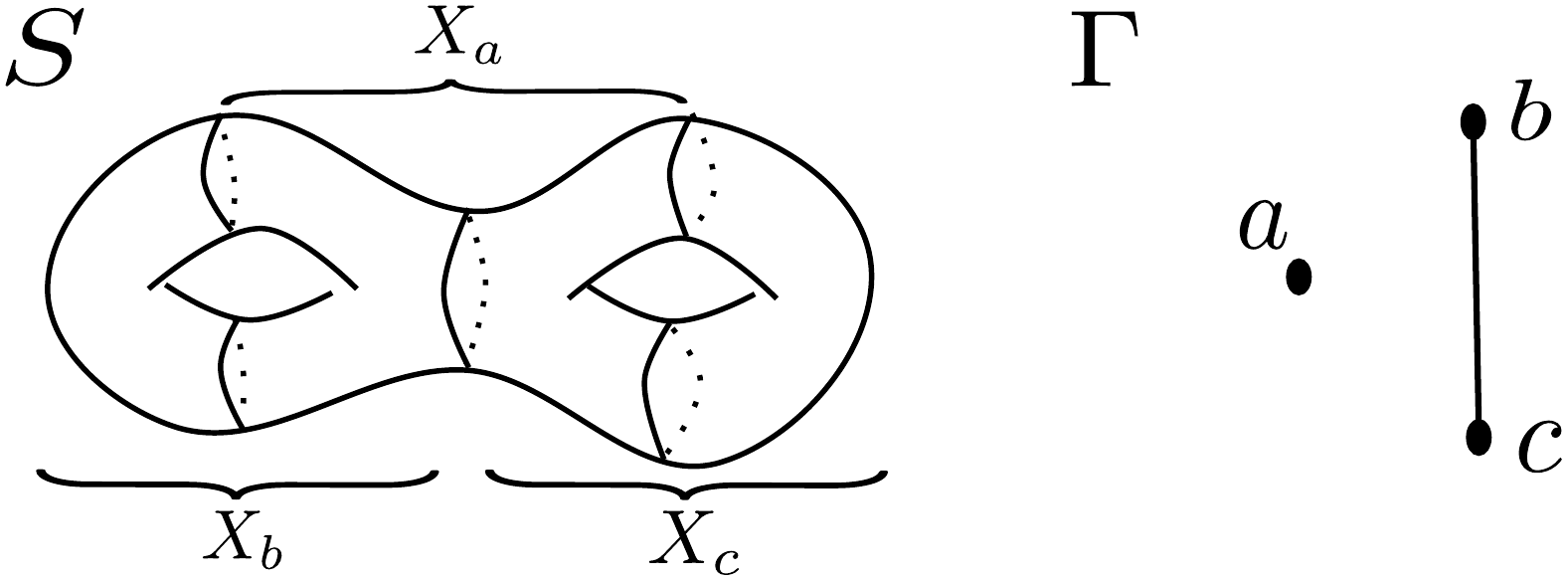}
\caption{Cleanly embedded subsurfaces with coincidence graph}
\label{fig1}
\end{center}
\end{figure}

\subsection{Mapping class group basics}\label{mcgbasics}
The mapping class group of the surface $S$ is the group of isotopy classes of orientation preserving homeomorphisms of $S$ and is denoted $\Mod(S)$. For a nonannular subsurface $X$ and $f \in \Mod(S)$, if $f$ has a representative homeomorphism that is the identity outside of $X$ then $f$ is said to be \emph{supported} on $X$. The mapping classes supported on $X$ are precisely those mapping classes in the image of the natural map $\Mod(X) \to \Mod(S)$ (see \cite{FM} for details). $f$ is \emph{fully supported} on $X$ if $f$ is supported on $X$ and the restriction of $f$ to $X$, denoted $f_X$, is pseudo-Anosov. Recall that $g \in \Mod(X)$ is \emph{pseudo-Anosov} if no positive power of $g$ fixes an isotopy class of curve in $X$. See \cite{FLP} for the standard definition of pseudo-Anosov and the equivalence to the definition given here. There is a natural action of $\Mod(S)$ on the curve complex $\C(S)$ obtained by extending the action of $\Mod(S)$ on isotopy classes of curves.

For $f \in \Mod(S)$, define the \emph{(stable) translation length} of $f$ on $\C(S)$ as follows: 
\[\ell_S(f)= \lim_{n\to \infty} \frac{d_{S}(\alpha, f^n(\alpha))}{n}.\]
where $\alpha$ is any vertex in $\V(S)$.  It is a standard exercise to verify that this limit exists and is independent of $\alpha$, and that $\ell_S(f^n) = n \cdot \ell_S(f)$. In \cite{MM1}, it is shown that if $f \in \Mod(S)$ is pseudo-Anosov then $\ell_S(f) \ge c >0$, where $c$ depends only on the topology of $S$. This implies that if $f$ is fully supported on the subsurface $X$ then $\ell_X(f) := \ell_X(f_X) > 0$. For more discussion on translation length see Section \ref{shortsec}.

Finally, recall that the Teichm\"{u}ller space of $S$, denoted $\mathcal{T}(S)$, is the space of marked hyperbolic structures on the surface $S$. We will consider $\mathcal{T}(S)$ with its Teichm\"{u}ller metric $d_{\mathcal{T}}$ and refer the reader to \cite{FM} for precise definitions and basic properties, including details about the action of $\Mod(S)$ on $\mathcal{T}(S)$ by isometries.  With this background at hand, we define the subgroups of the mapping class group that are of interest in this paper.

\begin{definition}
Let  $S$ be a surface with $\xi(S) \ge1$ and $H$ be a finitely generated subgroup of $\Mod(S)$. Then $H$ is \emph{convex cocompact} if for some $\mathcal{X} \in \mathcal{T}(S)$ the orbit $H\cdot \mathcal{X}$ is quasiconex (with respect to the Teichm\"{u}ller metric).
\end{definition}  

Farb and Mosher introduced convex cocompact subgroups of mapping class groups in \cite{FMo}, confirmed that the definition does not depend on the basepoint $\mathcal{X}$, and proved its equivalence to a detailed condition requiring both hyperbolicity of the subgroup and its well-behaved, cocompact action on a ``weak'' convex hull in $\T(S)$, justifying the analogy with convex cocompact Kleinian groups. Of particular relevance to this paper is the following characterization of convex cocompact subgroups of mapping class groups:

\begin{theorem}\cite{KL2,H}
Let  $S$ be a surface with $\xi(S) \ge1$ and $H$ be a finitely generated subgroup of $\Mod(S)$. Then $H$ is convex cocompact if and only if for some (any) $\alpha \in \C^0(S)$ the orbit map
$$H \to \C(S)$$
given by $h \mapsto h\cdot \alpha$ is a quasi-isometric embedding.
\end{theorem}

Using this characterization, it is immediate that if $H$ is convex cocompact then any $h \in H$ with infinite order is pseudo-Anosov.

\subsection{A partial order on subsurfaces}
We recall the partial order on connected subsurfaces described in \cite{CLM}. The origins of this ordering can be found in \cite{MM2, BKMM}. Fix markings $\mu, \mu'$ and fix $K \ge 20$. Define $\Omega(K,\mu,\mu') = \{Y \subset S: d_Y(\mu,\mu') \ge K \}$. Recall that connected subsurfaces $X,Y$ \emph{overlap} if $X$ and $Y$ intersect essentially and one cannot be isotoped to be contained in the other. This is equivalent to the condition that both $\pi_Y(\partial X)$ and $\pi_X(\partial Y)$ are not empty. Define the relation $X \prec Y$ to mean $X$ and $Y$ overlap and 
$$d_X(\mu, \del Y) \geq 10.$$ 
The following properties of $\prec$ are verified in \cite{CLM}.

\begin{lemma}\label{surfaceordering}
Let $K \ge 20$ and choose $X,Y \in \Omega(K,\mu,\mu')$ that overlap. Then $X$ and $Y$ are ordered and the following are equivalent
\begin{multicols}{2}
\begin{enumerate}
\item $X \prec Y$
\item $d_X(\mu,\partial Y) \ge 10$
\item $d_X(\mu, \partial Y) \ge K-4$
\item $d_X(\mu',\partial Y) \le 4$
\item $d_Y(\mu',\partial X) \ge 10$
\item $d_Y(\mu', \partial X) \ge K-4$
\item $d_Y(\mu,\partial X) \le 4$
\end{enumerate}
\end{multicols}
Moreover, $\prec$ is a strict partial order on $\Omega(K, \mu, \mu')$.
\end{lemma}

The following observation will be used in combination with the Bounded Geodesic Image Theorem (Theorem \ref{BGIT}) to show when large subsurface projections combine to build up distance in the curve complex.

\begin{lemma}\label{orderedtriple}
If $K \geq 20$, $X,Y,Z \in \Omega(K,\mu,\mu')$, and $X \prec Y \prec Z$, then any curve disjoint from both $X$ and $Z$ is also disjoint from $Y$.
\end{lemma}

\begin{proof}
Suppose $\gamma$ is disjoint from both $X$ and $Z$ but it intersects $Y$.  Because both $X$ and $Z$ are ordered with $Y$, they overlap with $Y$. So we have that $d_Y(\del X,\gamma), d_Y(\gamma, \del Z) \le 2$. Then
\begin{eqnarray*}
K \le d_Y(\mu, \mu') &\leq& d_Y(\mu, \del X) + d_Y(\del X,\gamma) + d_Y(\gamma, \del Z) + d_Y(\del Z, \mu')\\
 &\leq& 4 + 2 + 2 + 4 = 12
\end{eqnarray*}
using Lemma \ref{surfaceordering}. This contradicts $K \geq 20$.
\end{proof}

\section{Right-angled Artin groups}

\subsection{RAAGs and normal forms}\label{normalforms}
Let $\Gamma$ be a simplicial graph with vertex set $V(\Gamma) = \{v_1, \ldots, v_n \}$ and edge set $E(\Gamma) \subset V(\Gamma) \times V(\Gamma)$. The right-angled Artin group, $A(\Gamma)$, associated to $\Gamma$ is the group presented by
$$\langle v_i \in V(\Gamma): [v_i,v_j] =1 \iff (v_i,v_j) \in E(\Gamma) \rangle .$$
The $v_i$ will be referred to as the \emph{standard generators} of $A(\Gamma)$. Note that if $\Gamma$ is the graph with $n$ vertices and no edges then $A(\Gamma) = \F_n$. At the other extreme, if $\Gamma$ is the complete graph of $n$ vertices then $A(\Gamma) = \mathbb{Z}^n$. Because of this, $A(\Gamma)$ is often said to ``interpolate'' between free and free abelian groups. Although they are simple to define, right-angled Artin groups have been at the center of major recent developments in geometric group theory and low-dimensional topology.

In this section, we briefly recall a normal form for elements of a right-angled Artin group. For details see Section $4$ of \cite{CLM} and the references provided there. Fix a word $w = x_1^{e_1} \ldots x_k^{e_k}$  in the vertex generators of $A(\Gamma)$, with $x_i \in \{v_1, \ldots, v_n \}$ for each $i=1, \ldots, k$ . Each $x_i^{e_i}$ together with its index, which serves to distinguish between duplicate occurrences of the same generator, is a \emph{syllable} of the word $w$. Let $\syl(w)$ denote the set of syllables for the word $w$. We consider the following $3$ moves that can be applied to $w$ without altering the element in $A(\Gamma)$ it represents:
\begin{enumerate}
\item If $e_i = 0$, then remove the syllable $x_i^{e_i}$.
\item If $x_i = x_{i+1}$ as vertex generators, then replace $x_i^{e_i}x_{i+1}^{e_{i+1}}$ with $x_i^{e_i+e_{i+1}}$.
\item If the vertex generators $x_i$ and $x_{i+1}$ commute, then replace $x_i^{e_i}x_{i+1}^{e_{i+1}}$ with $x_{i+1}^{e_{i+1}}x_i^{e_i}$.
\end{enumerate}

For $\sigma \in A(\Gamma)$, set Min$(\sigma)$ equal to the set of words in the standard generators of $A(\Gamma)$ that have the fewest syllables among words representing $\sigma$. We refers to words in Min$(\sigma)$ as \emph{normal} representatives of $\sigma$. We also refer to a word in the standard generators as \emph{normal} if it is normal for the element of $A(\Gamma)$ that it represents. Hermiller and Meiler showed in \cite{HerM} that any word representing $\sigma$ can be brought to any word in Min$(\sigma)$ by application of the three moves above. Since these moves increase neither word nor syllable length, we see that words in Min$(\sigma)$ are also minimal length with respect to the standard generators, and that any two words in Min$(\sigma)$ differ by repeated application of move $(3)$ only. For normal words $w,w'$, we will occasionally use the notation $w \sim w'$ to denote that $w$ and $w'$ differ by a repeated application of move $(3)$. In other words, $w \sim w'$ if and only if there is a $\sigma \in A(\Gamma)$ with $w,w' \in \mathrm{Min}(\sigma)$. It is verified in \cite{CLM} that for any $\sigma \in A(\Gamma)$ and $w,w' \in \text{Min}(\sigma)$ there is a natural bijection between $\syl(w)$ and $\syl(w')$, which extends the obvious bijection between normal form words differing by a single application of move $(3)$.  Thus we define, for $\sigma \in A(\Gamma)$, $\syl(\sigma) := \syl(w)$ using any $w \in \text{Min}(\sigma)$. This permits us to define, for each $\sigma \in A(\Gamma)$, a strict partial order $\prec$ on the set $\syl(\sigma)$ by setting $x_i^{e_i} \prec x_j^{e_j}$ if and only if for every $w \in \text{Min}(\sigma)$ the syllable $x_i^{e_i}$ appears to the left of $x_j^{e_j}$ in the spelling of $w$.

One can imagine that the generators and their inverses correspond to directions in the Cayley graph.  For example, the standard Cayley graphs for either $F_2 = \grp{a,b}$ or $\Z^2 = \grp{a,b|ab=ba}$ have four directions; typically $a$ and $a^{-1}$ correspond to east and west respectively, while $b$ and $b^{-1}$ point north and south.  A word in normal form represents a geodesic path that also minimizes changes in direction.

\begin{example}
Take $\Gamma$ to be the graph with vertex set $\{a,b,c,d\}$ and edge set $\{(a,b),(b,c),(c,d)\}$ and consider the word $w = a c b  d$. We see that $w$ is in normal form and is equivalent through move $(3)$ to the words $a b c  d$, $b a c  d$, $a b d  c$, and $b a d  c$. Hence, the only ordered syllables of $w$ are $a \prec c$, $a\prec d$, and $b \prec d$.
\end{example}

Our first lemma states immediate properties of our partial order and normal forms.  We use this lemma without mention, in particular to prove a second, more technical lemma critical to our proof of the main theorem.  

\begin{lemma}\label{subwordsobvious}Subwords of normal words are themselves normal words.  Unordered syllables of a normal word correspond to generators that commute.
\end{lemma}

\begin{lemma} \label{subwords}
Suppose $w$ is a normal word containing distinct, unordered syllables $p,q \in \syl(w)$ separated by the subword $M$.  That is, up to switching $p$ and $q$,
\[w = x_1^{e_1}\cdots x_k^{e^k} \qquad p=x_i^{e_i} \qquad q=x_j^{e_j} \qquad M = x_{i+1}^{e_{i+1}}\cdots x_{j-1}^{e_{j-1}},\]
so $pMq$ is the smallest subword of $w$ containing $p$ and $q$.  Then $M$ has normal representative $M' = LR$ where $L$ and $R$ are (possibly empty) subwords commuting with $p$ and $q$ respectively.
\end{lemma}

\begin{proof}We induct on the syllable length of $M$.  The claim is vacuous if $M$ is the empty word.  If $M$ non-empty, observe that each syllable of $M$ is ordered with at most one of $p$ and $q$, since $p$ and $q$ are not ordered.  Find the first syllable in $M$ from the left which is ordered with $p$.  If no such syllable exists then the claim is true for $M=L$ and $R$ empty.  Otherwise, call this syllable $s$ and observe that $s$ and $q$ are not ordered and therefore commute.  Furthermore, $pMq = pL_1sM_1q$ where $L_1$ commutes with $p$ by construction.  Inductively, $M_1$ has normal representative $L_2R_2$ where $L_2$ commutes with $s$ and $R_2$ commutes with $q$.  Thus $pMq$ has normal representative $L_1pL_2qsR_2$ got by a repeated application of move $(3)$.  Inductively, $L_2=L_3R_3$ where $p$ commutes with $L_3$ and $q$ commutes with $R_3$.  We have shown that the word $M$ has normal representatives
\[M \sim L_1sM_1 \sim L_1sL_2R_2 \sim L_1L_2sR_2 \sim L_1L_3 \cdot R_3sR_2\]
 where $L=L_1L_3$ commutes with $p$ and $R=R_3sR_2$ commutes with $q$.\end{proof}

\subsection{Cubical geometry}\label{cubegeo}
In this section we briefly review some geometry of non-positively curved cube complexes, which we use throughout the paper. Good references for this material are \cite {BH, HW1, Hag}. First, recall that a simplicial complex is a \emph{flag complex} if any $n+1$ pairwise adjacent vertices span an $n$-simplex. A subcomplex $X$ of a simplicial complex $Y$ is a \emph{full subcomplex} when any simplex in $Y$ whose vertices are in $X$ is contained in $X$. 

\begin{definition}
A \emph{cube complex} $X$ is the space formed by isometrically gluing Euclidean unit cubes along their faces. $X$ is \emph{non-positivley curved (NPC)} if the link of each vertex is a flag simplicial complex. $X$ is \emph{CAT$(0)$} if it is \emph{NPC} and simply connected.
\end{definition}

Because of their combinatorial nature, local isometries between NPC cube complexes have a particularly simple description; see \cite {Char, BH, Hag} for details. A map between cube complexes is said to be \emph{cubical} if it maps open cubes homeomorphically onto open cubes. Denote by $\lk(x)$ the link of the vertex $x$ of $X$.

\begin{theorem}\label{localiso}
Let $X, Y$ be cube complexes, $Y$ \emph{NPC}, and $f: X \to Y$ a cubical map such that for each $x \in X^0$ the induced map on the link of $x$ is injective and $f(\lk(x))$ is a full subcomplex of $\lk(f(x))$ in $Y$. Then $X$ is \emph{NPC}, and $f$ is a local isometry. Further, $f_*: \pi_1(X) \to \pi_1(Y)$ is injective and the induced map on the universal covers $\tilda{f}: \tilda{X} \to \tilda{Y}$ is an isometric embedding with $\tilda{f}(\tilda{X}) \subset \tilda{Y}$ a convex subcomplex.
\end{theorem}

Given a simplical graph $\Gamma$ with associated right-angled Artin group $A(\Gamma)$, we recall the definition of the so-called \emph{Salvetti complex} $S_{\Gamma}$. This is the cube complex defined as follows: begin with a rose, denoted $S^1_{\Gamma}$, with vertex $x$ and $|V(\Gamma)|$ petals oriented and labeled by the vertices of $\Gamma$. Now attach $2$-cubes for each edge of $\Gamma$ corresponding to a commutation relator in $A(\Gamma)$. Specificaly, if $u$ and $w$ are vertices of $\Gamma$ joined by an edge, then a square is attached to $S^1_{\Gamma}$ with boundary label $uwu^{-1}w^{-1}$.  For $n > 2$, we attach an $n$-cube for each set of $n$ pairwise commuting generators (for each $n$-clique of $\Gamma$) whose attaching map restricted to each face is the characteristic map for that face. It is easy verified that $S_{\Gamma}$ is an NPC cube complex with fundamental group $A(\Gamma)$, see \cite{BH} for example. We denote its universal cover $\tilda{S}_{\Gamma}$ and fix a lift $\tilda{x}$ of the unique vertex $x$ of $S_{\Gamma}$. By the construction, it follows that ($\tilda{S}^1_{\Gamma}, \tilda{x})$ is isomorphic to the Cayley graph for $A(\Gamma)$ with its standard generating set. In this paper, our cube complexes arise as compact locally convex subcomplexes of covers of the Salvetti complex. These are known as compact \emph{special} cube complexes in the literature. 

Besides the induced path metric on a CAT$(0)$ cube complex $X$, we will also be interested in the graph metric on the $1$-skeleton of $X$. We refer to this metric as \emph{combinatorial distance} on $X^0$ and observe that this distance can be alternatively characterized as the number of hyperplanes separating two vertices or the number of hyperplanes intersected by the CAT$(0)$ geodesic between the vertices. Geodesics in this metric on $X^1$ are called \emph{combinatorial geodesics}. For more on combinatorial distance see \cite{Hag}, or the appendix of \cite{HW1}. In the special case of $\tilda{S}_{\Gamma}$, we note that combinatorial distance agrees with distance in $A(\Gamma)$ with its standard generating set via the identification $A(\Gamma) \to \tilda{S}_{\Gamma}$ given by $g \mapsto g \cdot \tilda{x}$.

Finally, recall that a collection of vertices $Y \subset X^0$ for $X$ a CAT$(0)$ cube complex is called \emph{combinatorially $K$-quasiconvex} in $X$ if \emph{every} combinatorial geodesic between points of $Y$ stays within combinatorial distance $K$ from $Y$. Again translating to the group $A(\Gamma)$, quasiconvexity of the subgroup $H < A(\Gamma)$ with respect to the standard generators is equivalent to combinatorial quasiconvexity of $H\cdot\tilda{x}$ in $\tilda{S}_{\Gamma}$.

\subsection{Quasiconvex subgroups of RAAGs} \label{quasi}
We will make use of the following theorem of Haglund. Here, the \emph{convex hull} of a subcomplex $Y$ of the CAT$(0)$ cube complex $X$ is the intersection of all convex subcomplexes containing $Y$.

\begin{theorem}\cite{Hag} \label{Haglund}
Let $X$ be a uniformly locally finite $\mathrm{CAT}(0)$ cube complex. Then for any $K\ge0$ there exists an $L \ge 0$ such that the convex hull of any combinatorially $K$-quasiconvex subcomplex $Y$ is contained in the $L$-neighborhood of $Y$.
\end{theorem}

We recall that a subcomplex $Y \subset X$ is CAT($0$) convex if and only if $Y$ is full and combinatorially convex \cite{Hag, HW1}. For a simplicial  graph $\Gamma$ and $K \ge 0$, define the function $f(\Gamma, K)$ as follows: first use Theorem \ref{Haglund} to choose an $L$ so that the $L$-neighborhood of any combinatorially $K$-quasiconvex subcomplex $Y$ of the CAT$(0)$ cube complex $\tilda{S}_{\Gamma}$ contains the convex hull of $Y$. Now define $f(\Gamma, K)$ to be the number of vertices of combinatorial distance $\le L$ from $\tilda{x} \in \tilda{S}_{\Gamma}$, where $\tilda{x}$ is a lift of the unique vertex of $S_{\Gamma}$. Hence, $f(\Gamma, K)$ counts the number of elements in the $L$-ball about the identity in the Cayley graph of $A(\Gamma)$ with its standard generators. We now obtain the following consequence of Theorem \ref{Haglund}.

\begin{lemma}\label{Cexists}
Let $H < A(\Gamma)$ be a quasiconvex subgroup with respect to the standard generators of $A(\Gamma)$. Then there exists a pointed compact cube complex $(C,x)$ and a cubical local isometry $\phi: (C,x) \to (S_{\Gamma},x)$ with $H= \phi_*(\pi(C,x))$. Furthermore if $H$ is $K$-quasiconvex, then $C$ has less than $f(\Gamma,K)$ vertices.
\end{lemma}

\begin{proof}
Since quasiconvexity of $H$ in the standard generators of $A(\Gamma)$ is equivalent to combinatorial quasiconvexity of $H\cdot\tilda{x}$, Theorem \ref{Haglund} provides a convex subcomplex $\tilda{C}$ containing $\tilda{x}$, the convex hull of $H\cdot\tilda{x}$ in $\tilda{S}_{\Gamma}$, that is $H$-invariant and cocompact. Set $C = \tilda{C} / H$ and let $\phi: C \to S_{\Gamma}$ denote the composition $\tilda{C} /H \to \tilda{S}_{\Gamma} /H \to S_{\Gamma}$. As the composition of local isometries, $\phi$ is a local isometry and $H= \phi_*(\pi(C,x))$.

To conclude, note that if $H$ is $K$-quasiconvex then the number of vertices of $C$ is no greater than the number of vertices of $N_L(H \cdot \tilda{x}) / H$, as $\tilda{C} \subset N_L(H\cdot\tilda{x})$. This is in turn bounded by the number of vertices of $N_L(\tilda{x})$ in $\tilda{S}_{\Gamma}$, because the projection map $N_L(\tilda{x}) \to N_L(H\cdot\tilda{x}) / H$ is surjective. Hence, the number of vertices of $C$ is bounded above by $f(\Gamma,K)$.
\end{proof}

Fixing a lift $\tilda{x} \in \tilda{S}_{\Gamma}$ over $x \in S_{\Gamma}$, the map $\phi$ lifts to a cubical isometric embedding $\tilda{\phi}: (\tilda{C},\tilda{x}) \to (\tilda{S}_{\Gamma}, \tilda{x}).$ Through this embedding, $\tilda{C}$ becomes a convex, combinatorially convex subcomplex of $\tilda{S}_{\Gamma}$.

Label the oriented edges of $S_{\Gamma}$ by the vertex generators of $A(\Gamma)$ as in the construction of $S_{\Gamma}$. Using the cubical map $\phi$, pull back these labels to the oriented edges of $C$ so that oriented loops in $C$ at $x$ correspond to words in $H$ written in the standard generators of $A(\Gamma)$. We note the following implication of the geometric properties of $C$. First, call an edge path in a nonpositively curved cube complex $X$ a \emph{combinatorial local geodesic} if each of its lifts to $\tilda{X}$ is a combinatorial geodesic. Note that this property is preserved under local isometry of cube complexes.

\begin{lemma}\label{geowords}
With the notation above, oriented edge-loops at $x \in C$ which are combinatorial local geodesics are in bijective correspondence with minimal length words in $H$ with respect to the standard generators of $A(\Gamma)$.
\end{lemma}

\begin{proof}
By Lemma \ref{Cexists}, a loop based at $x$ in $S_{\Gamma}$ corresponding to a word in $H$ is a combinatorial local geodesic in $S_{\Gamma}$ if and only if it is contained in $C$ and is a combinatorial local geodesic there. More precisely, if $\gamma$ is a combinatorial local geodesic loop of $(S_{\Gamma},x)$ with $[\gamma] \in H$, then its lifts based at $\tilda{x}$ in $\tilda{S}_{\Gamma}$ are contained in $\tilda{C}$, and so $\gamma$ corresponds to a combinatorial local geodesic in $C$. Conversely, combinatorial local geodesic loops of $(C,x)$ map through $\phi$ to combinatorial local geodesic loops of $(S_{\Gamma},x)$ representing elements of $H$.

The lemma then reduces to the observation that since $(\tilda{S}_{\Gamma}^1,\tilda{x})$ is isomorphic to the Cayely graph of $A(\Gamma)$ with the standard vertex generators, combinatorial local geodesics in $S_{\Gamma}$ correspond to minimal length words in $A(\Gamma)$.
\end{proof}

The following proposition gives us control over the normal forms of elements of $H$.

\begin{proposition} \label{quasicontrol}
Let $H$ be a $K$-quasiconvex subgroup of $A(\Gamma)$. There is an $\ell=\ell(\Gamma,K) \ge 0$ so that for $h \in H$ and a minimal length representative $w =w_1w_2 \ldots w_n$ of $h$ in the standard generators of $A(\Gamma)$, any subword  $w_iw_{i+1} \ldots w_j$ of $w$ with length $j-i \ge \ell/3$ contains as a subword a conjugate of an element of $H$ of $A(\Gamma)$-length $\le \ell$.
\end{proposition}

\begin{proof}
Since $H$ is quasiconvex in $A(\Gamma)$ with the standard generators, $H\cdot\tilda{x}$ is combinatorially quasiconvex in $\tilda{S}_{\Gamma}$. Hence, there is a cube complex $C$ and a cubical local isometry $\phi: (C,x) \to (S_{\Gamma},x)$ with $H = \phi_*(\pi_1(C,x))$, as discussed above. Set $\ell = 3 (f(\Gamma,K) +1)$.

If $w$ is as in the statement of the proposition, then by Lemma \ref{geowords} above $w$ corresponds to a combinatorial local geodesic in $C$. Since subpaths of combinatorial geodesics in $\tilda{C}$ are themselves combinatorial geodesics, the subword $w_iw_{i+1} \ldots w_j$ is also a combinatorial local geodesic of $C$. This implies, in particular, that the edge path in $C$ spelled by $w_iw_{i+1} \ldots w_j$ does not backtrack in $C^1$. If $j-i \ge f(\Gamma,K)+1$, then this edge path must repeat a vertex, say $w \in C$, since the number of vertices of $C$ is no greater then $f(\Gamma,K)$ by Lemma \ref{Cexists}.

Let $\alpha$ be the subpath of $w_iw_{i+1} \ldots w_j$ beginning and ending at $w$. Again, $\alpha$ is a combinatorial local geodesic loop of length $\le f(\Gamma,K)+1$. Since the combinatorial distance from $x$ to $w$ is also $\le f(\Gamma,K)+1$, we have that $\alpha$ represents a word in $A(\Gamma)$ conjugate to a loop in $C$ based at $x$ of length $\le 3(f(\Gamma,K)+1) = \ell$. This complete the proof.
\end{proof}

For concrete examples of the local isometries $C \to S_{\Gamma}$ considered above, see Section \ref{sec:examples}.

\section{Admissible $A(\Gamma)$ in $\Mod(S)$}\label{sec:admissible}

\subsection{The \cite{CLM} construction}
In order to introduce notation and motivate the condition to which our theorem applies, we begin by describing the types of the homomorphisms constructed in \cite{CLM}. Begin with a finite simplicial graph $\Gamma$ and a collection of cleanly embedded, connected nonannular subsurfaces $\mathbb{X} =\{X_1, \ldots, X_n\}$ in $S$, whose coincidence graph is $\Gamma = \Gamma_{\mathbb{X}}$. We label the vertices of $\Gamma$ by $v_1, \ldots, v_n$ so that these indices agree with subsurfaces they represent in $\mathbb{X}$. Since mapping classes supported on disjoint subsurfaces commute, any choice of mapping classes $\mathbb{F} =\{f_1, \ldots, f_n \}$ with $f_i \in \Mod(X_i) \le \Mod(S)$ determines a homomorphism $\phi = \phi_{\mathbb{F}}: A(\Gamma) \to \Mod(S)$ with $\phi(v_i) = f_i$. Suppose we have fixed such a homomorphism where the mapping classes are fully supported on their respective surfaces. Recall this implies that for each $1\le i\le n$, $\ell_{X_i}(f_i) > 0$. Informally, the main theorem from \cite{CLM} concludes that if these translation lengths are large enough, depending only on $\mathbb{X}$, then the induced homomorphism into $A(\Gamma)$ is a quasi-isometric embedding. Here are some of the relevant details.

For $\sigma \in A(\Gamma)$ with $w = x_1^{e_1} \dots x_k^{e_k} \in \text{Min}(\sigma)$, define $J: \{1, \ldots, k\} \to \{1, \ldots, n \}$ so that $x_i = v_{J(i)}$ as standard generators of $A(\Gamma)$. Hence, $\phi(x_i) = f_{J(i)}$ is supported on $X_{J(i)}$. Write
$$X^w(x_i^{e_i})  = \phi(x_1^{e_1} \ldots x_{i-1}^{e_{i-1}})(X_{J(i)})$$
for $i = 2, \ldots, k$ and $X^w(x_1^{e_1}) = X_{J(1)}$. This defines a map
$$X^w : \text{syl}(w) \to \Omega(S) .$$ 
It is verified in \cite{CLM} that this map is well-defined for $\sigma \in A(\Gamma)$ independent of the choice of normal representative $w$, so we set $X^{\sigma} := X^w$ for $w \in \text{Min}(\sigma)$.

Define $\subs(\sigma)$ as the image of the map $X^{\sigma}: \mathrm{syl}(\sigma) \to \Omega$.  Thus we may associate to every syllable $s$ in $\syl(\sigma)$ the subsurface $Y=X^{\sigma}(s)$ in $\subs(\sigma)$, without reference to the particular indexing of syllables in $w$.  The subsurface $Y$ is not to be confused with the support of the syllable $\phi(s) = f_j^i$, denoted $\supp(\phi(s))$, which is $X_j$.

\begin{example}
Consider the surface $S$ and the collection of subsurfaces $\mathbb{X} = \{ X_a, X_b, X_c\}$ given in Figure \ref{fig1} with coincidence graph $\Gamma$. Let $\F = \{f_a,f_b,f_c \}$ be a collection of mapping classes that are fully supported on $\mathbb{X}$ and let $\phi: A(\Gamma) \to \Mod(S)$ be the induced homomorphism. Consider the normal form word $w=abca$. Then $\subs(w) = \{X_a, f_aX_b, f_af_bX_c,$  $f_af_bf_cX_a \}$. A single application of move $(3)$ yields the word $w'=acba$ with $\subs(w') = \{X_a,f_aX_c, f_af_cX_b, f_af_cf_bX_a \}$. As $X_b$ and $X_c$ are disjoint, $f_b$ and $f_c$ commute and so $\subs(w) = \subs(w')$, as claimed above. 
\end{example}

\subsection{Admissibility}
Suppose that we have fixed the homomorphism $\phi_{\mathbb{F}}: A(\Gamma) \to \Mod(S)$ as in the previous section. The following definition collects the properties necessary for the proof of our main theorem.

\begin{definition}  \label{admisdef}
Fix the marking $\mu \in \mathcal{M}(S)$. Let $\Gamma$ be the coincidence graph of  a collection $\mathbb{X} =\{X_1, \ldots, X_n\}$ of cleanly embedded, nonannular, proper subsurfaces of $S$.  Choose $\mathbb{F}$ a collection of mapping classes fully supported on $\mathbb{X}$. The homomorphism $\phi = \phi_{\F}: A(\Gamma) \to \text{Mod}(S)$ is K-\emph{admissible} if the following conditions holds for each $\s \in A(\Gamma)$:
\begin{itemize}

\item For each syllable $x_i^{e_i}$ of $\s$, we have $d_{X^{\s}(x_i^{e_i})} (\mu,\phi(\s) \mu) \ge K|e_i|$. In particular, $X^{\s}(x_i^{e_i}) \in \Omega(K,\mu, \phi(\s)\mu)$.

\item The map $X^{\s}: \mathrm{syl}(\sigma) \to \subs(\sigma)$ is injective and order preserving, where $\subs(\sigma) \subset \Omega(K,\mu, \phi(\s)\mu)$ is given the induced ordering.

\end{itemize}
\end{definition}

The main result from \cite{CLM} can then be rephrased as

\begin{theorem}\cite{CLM}\label{CLMtheorem}
Let $\mathbb{X} = \{X_1, \ldots, X_n \}$ be cleanly embedded in $S$ with co-incidence graph $\Gamma_{\mathbb{X}}$. Then for any $K$ there is a $c_0 \ge 0$ so that the following holds: if $\F = \{f_1, \ldots, f_n\}$ is a collection of mapping classes fully supported on $\mathbb{X}$ and for each $f_i \in \Mod(X_i)$ we have $\ell_{X_i}(f_i) \ge c_0$, then the induced homomorphism $\phi_{\F}: A(\Gamma) \to \Mod(S)$ is $K$-admissible. For $K$ sufficiently large, any $K$-admissible homomorphism is a quasi-isometric embedding into $\Mod(S)$ and has a quasi--isometric orbit map into Teichm\"{u}ller space.
\end{theorem}

We remark that while in the proof of the Theorem \ref{CLMtheorem} showing that large translation length implies admissibility is quite involved, once admissibility is demonstrated, showing that the homomorphisms are quasi-isometric embeddings into $\Mod(S)$ and Teichm\"{u}ller space follows immediately from the distance formulas of Masur-Minsky \cite{MM2} and Rafi  \cite{Rafi}, respectively. In \cite{Tayl1}, the second author gives a similar criterion for quasi-isometric embeddings of right-angled Artin groups into $\Out(F_n)$.

\begin{definition}
Assuming the set-up of Definition \ref{admisdef}, $\phi_{\F}: A(\Gamma) \to \Mod(S)$ is \emph{admissible} if it is $K$-admissible for some $K \ge \max\{K_{BGI},20\}$, with $K_{BGI}$ as in Theorem \ref{BGIT}.
\end{definition}

\subsection{Filling words, filling blocks, and $\ell$-short filling subgroups}
Fix an admissible homomorphism $\phi_{\mathbb{F}}: A(\Gamma) \to \Mod(S)$.  In this section we define three notions of ``filling,'' for words, subwords, and subgroups in $A(\Gamma)$ respectively.  We emphasize that these conditions depend on the particular admissible embedding of $A(\Gamma)$ into $\Mod(S)$, which we denote by $\phi$ for brevity.  These definitions allow us to develop the technical machinery used in the proof of Theorem \ref{main}, from which we derive the ``if'' direction of Theorem \ref{intromain}: that quasiconvex, all-pseudo-Anosov subgroups of admissibly embeddings $A(\Gamma)$ are convex cocompact.

Call $\sigma \in A(\Gamma)$ \emph{cyclically reduced} if it has the least number of syllables among all of its conjugates. If $\sigma$ is cyclically reduced, we say that $\sigma$ \emph{fills} if the collection of subsurfaces $\subs(\sigma)$ fills $S$; that is, for any curve $\gamma$ there is a subsurface $Y \in \subs(\sigma)$ such that $\pi_Y(\gamma)$ is not empty.  For arbitrary $\sigma \in A(\Gamma)$, we say that $\sigma$ fills if $\sigma$ is conjugate to $\sigma'$ where $\sigma'$ is cyclically reduced and fills.  Note that this is equivalent to requiring that all conjugates $\sigma'$ of $\sigma$ are such that the collection of subsurfaces $\subs(\sigma')$ fills $S$.  We say a word $w$ in the standard generators of $A(\Gamma)$ \emph{fills}, or is a \emph{filling word}, if and only if $w \in \Min(\sigma)$ for some $\sigma \in A(\Gamma)$ that fills.

Filling words relate to pseudo-Anosov mapping classes in an unsurprising way.  The idea for the next lemma appears in \cite{CLM}, but we include a proof for completeness.  Let $\mathbb{X}_\sigma$ be the collection of subsurfaces \[\mathbb{X}_\sigma = \{X \in \mathbb{X}: X = \supp(\phi(x_i)) \text{ for some }x_i^{e_i}\in \syl(\sigma)\}.\]

\begin{lemma}\label{subsfills}  The set of subsurfaces $\subs(\sigma)$ fills $S$ if and only if $\mathbb{X}_\sigma$ fills $S$.\end{lemma}

\begin{proof}Fix an indexing by $i$ of syllables $x_i^{e_i}$ in $\syl(\sigma)$, their supports $X_i$ in $\mathbb{X}$, and their associated subsurfaces $Y_i = \phi (x_1^{e_1}\cdots x_{i-1}^{e_{i-1}})X_i$ in $\subs(\sigma)$.  If $\mathbb{X}_\sigma$ fills $S$, then given a curve $\gamma$ we may consider the smallest index $I$ such that $\gamma$ cuts $X_I$.  Because $\phi(x_1^{e_1}\cdots x_{I-1}^{e_{I-1}})$ fixes $\gamma$, we know $\gamma = \phi(x_1^{e_1}\cdots x_{I-1}^{e_{I-1}})(\gamma)$ cuts $Y_I = \phi(x_1^{e_1}\cdots x_{I-1}^{e_{I-1}})X_I$.  Since $\gamma$ was arbitrary, we conclude $\subs(\sigma)$ fills $S$.  On the other hand, if some curve $\gamma$ misses all $X_i$ in $\mathbb{X}_\sigma$, then because $\gamma = \phi( x_1^{e_1}\cdots x_{i-1}^{e_{i-1}})(\gamma)$ for any $i$, it is easy to see that $\gamma$ is also disjoint from all $Y_i$ in $\subs(\sigma)$.\end{proof}

Lemma \ref{subsfills} implies that, if $\subs(\sigma)$ fails to fill, $\phi(\sigma)$ fixes the isotopy class of some curve and thus cannot be pseudo-Anosov.  Therefore, for any $\sigma \in A(\Gamma)$, if $\phi(\sigma)$ is pseudo-Anosov, then $\sigma$ fills.  The converse is shown in \cite{CLM}: if $\s$ fills then $\phi(\s)$ is pseudo-Anosov.  However, we do not directly use this fact, which can be recovered from Theorem \ref{main} in the case $H$ is a cyclic group.  Indeed, rather then assume that $\phi(H)$ is a purely pseudo-Anosov subgroup of the mapping class group, Theorem \ref{main} requires the (a priori weaker) filling assumption on a finite collection of elements of the subgroup $H$.  We formalize this assumption with a definition: say that for $H \le A(\Gamma)$ and $\ell >0$, $H$ is \emph{$\ell$-short filling} if for all $h \in H$ with $|h|_{A(\Gamma)} \le \ell$, $h$ fills. We remind the reader that this condition is for a fixed admissible homomorphism $\phi: A(\Gamma) \to \Mod(S)$.

Consider a normal word $w = s_1\cdots s_k$ where $s_i$ are syllables in $\syl(w)$.  Suppose $w'$ is a subword consisting of a sequence of consecutive syllables, so $w' = s_is_{i+1}\cdots s_{j}$ where $i < j$.  If $\subs(w')$ fills $S$, we say $w'$ is a \emph{filling block} for $w$.  Note that, while every filling word has a filling block, filling blocks can also appear in non-filling words.  We are interested in filling blocks because, as Lemma \ref{blockcommutation} shows, these have consequences about syllable order.  Theorem \ref{CLMtheorem} relates syllable order to subsurface order, and we prove Theorem \ref{main} by relating subsurface order to distances in the curve complex.

The proof of Theorem \ref{main} relies on the lemmas that follow.  We have already mentioned Lemma \ref{blockcommutation}, which shows that filling blocks block commutation.  Lemma \ref{blockexistence} show how to find filling blocks for $\ell$-short filling $H$.  Before either of these, we need to relate $\subs(w')$ to $\subs(w)$ when $w'$ is a subword of $w$.  Given $u,v \in A(\Gamma)$, write $u\cdot\subs(v)$ to denote the collection $\{\phi(u)Y: Y\in \subs(v)\}$.

\begin{lemma}\label{subwordsubs} 
Suppose $w$ is a normal word with length $k$ in the standard generators of $A(\Gamma)$, partitioned into subwords by $w=BME$, so 
\[w = x_1\cdots x_k \qquad B = x_1\cdots x_i \qquad M = x_{i+1}\cdots x_j \qquad 1 \leq i < j \leq k\]
where $x_i$ are generators of $A(\Gamma)$, and $E$ may be empty.  Then $B\cdot\subs(M) \subset \subs(w)$, and any syllable $s' \in \syl(M)$ with associated subsurface $Y \in \subs(M)$ is a subword of the syllable $s \in \syl(w)$ with associated subsurface $\phi(B)Y$. 
\end{lemma}

\begin{proof}This follows by inspection, using Lemma \ref{subwordsobvious} and the definition of $\subs$.\end{proof}

\begin{lemma}\label{blockcommutation} 
Suppose the normal word $w = s_1\cdots s_k$ contains a filling block $w'= s_is_{i+1}\cdots s_{j}$, where $s_n$ are syllables in $\syl(w)$ for $1\le n \le k$.  Then each syllable $s_n$ is ordered with some syllable of $w'$.
\end{lemma}

\begin{proof}Let $Y \in \subs(w)$ be the subsurface associated to the syllable $s_n$, that is $Y = X^w(s_n)$.  Write $w=Bw'E$ where $B$ and $E$ are the prefix and suffix respectively of the subword $w'$ in $w$.  Because $\subs(w')$ fills, so does its homeomorphic image $B \cdot \subs(w')$, which is a subset of $\subs(w)$ by Lemma \ref{subwordsubs}.  In particular we have $Z \in B \cdot \subs(w') \subset \subs(w)$ such that $\pi_Z(\del Y)$ is not empty.  Let $s_m$ be the syllable of $w$ associated to $Z$; observe that $s_m$ is also a syllable of $w'$.  Write $X_n$ and $X_m$ for the supports of $s_n$ and $s_m$ respectively.

We will show that $s_n$ and $s_m$ are ordered.  Suppose not.  Then $w$ has normal representatives $w_1 = bs_ns_me$ and $w_2 = bs_ms_ne$ which each represent some $\sigma \in A(\Gamma)$.  
Since either of $X^{w_1}$ and $X^{w_2}$ can be used to determine the bijection $X^{\sigma}: \syl(\sigma) \to \subs(\sigma)$, we know
\[Y = \phi(b)X_n = \phi(bs_m)X_n \qquad \text{ while } \qquad Z = \phi(bs_n)X_m = \phi(b)X_m.\]
This means $\phi(s_n)X_m=X_m$, which requires that $X_m$ and $X_n$ are either the same or they are disjoint subsurfaces, by admissibility.  Either case contradicts our finding that
\[\pi_Z(\del Y) = \pi_{\phi(b)X_m}(\del (\phi(b)X_n)) = \phi(b)(\pi_{X_m}(\del X_n))\]
is not empty.\end{proof}

Finally, we give lemmas to find filling blocks whenever $H < A(\Gamma)$ is quasiconvex and $\ell$-short filling for appropriate $\ell$.

\begin{lemma} \label{shortsyllables}
Given an admissible embedding $\phi: A(\Gamma) \to \Mod(S)$, suppose $H < A(\Gamma)$ is quasiconvex and $\ell$-short filling, where $\ell$ is the constant determined in Proposition \ref{quasicontrol}.  Then for any $h \in H$, every syllable $x^e \in \syl(h)$ has exponent $e < \ell/3$.\end{lemma}

\begin{proof}
Supposing the contrary, let $w \in \mathrm{Min}(h)$ with syllable $x^e$. Since normal form words are length minimizing, Proposition \ref{quasicontrol} says that $x^e$ contains a subword conjugate to $h' \in H$ where $|h'|_{A(\Gamma)} \leq \ell$.  Then by hypothesis, $h'$ fills.  Thus all its conjugates fill, including the subword of $x^e$.  But any such subword is merely a power of the generator $x$, whose image under the admissible embedding $\phi$ is supported on a proper subsurface of $S$.  In particular, such a subword cannot fill, our contradiction.
\end{proof}

\begin{lemma}\label{blockexistence} 
Given an admissible embedding $\phi: A(\Gamma) \to \Mod(S)$, suppose $H < A(\Gamma)$ is quasiconvex and $\ell$-short filling, where $\ell$ is the constant determined in Proposition \ref{quasicontrol}.  Then for any $h \in H$ and any $w \in \mathrm{Min}(h)$, every subword of $w$ of length at least $\ell$ contains a filling block for $w$.
\end{lemma}

\begin{proof}
Let $w'$ be a subword of $w$ of length at least $\ell$.  Then $w'$ partitions into subwords $w' = BME$ where $M$ has length $\ell/3$ while $B$ and $E$ have length at least $\ell/3$.  By Proposition \ref{quasicontrol}, $M$ has a subword $M'$ which, as an element of $A(\Gamma)$, is conjugate to $h' \in H$ such that $|h'|_{A(\Gamma)} \leq \ell$.  By the hypothesis that $H$ is $\ell$-short filling, $h'$ is a filling word, and therefore so is its conjugate $M'$.  This means $\subs(M')$ fills.  If $M'$ is not already a filling block for $w$, it must be that the first and/or last syllables of $M'$ are proper subwords of syllables of $w$.  By Lemma \ref{shortsyllables}, since $B$ and $E$ have length greater than $\ell/3$, it is possible to complete $M'$ to a filling block of $w$ which is entirely contained in $w'$.
\end{proof}


\section{Proof of convex cocompactness}\label{sec:proofsufficient}

In this section, we provide the conditions for when a subgroup of the mapping class group that come from an admissible embedding of a right-angled Artin group is convex cocompact. The idea is the following: fix a marking $\mu$ and begin with a normal word $w$ of $H \le A(\Gamma)$. If $w$ has $n$ syllables, admissibility implies that there are at least $n$ subsurfaces with large subsurface projection between $\mu$ and its translate $\phi(w)\mu$.  An application of the Bounded Geodesic Image Theorem (Theorem \ref{BGIT}) implies that any geodesic between the marking and its translate in $\C(S)$ must enter a $1$-neighborhood in the curve complex of the boundary of each of the subsurfaces of large projection distances. This will force curve complex distance of at least $\frac{n}{k}$ so long as we have a $k$-to-1 association between subsurfaces of large projection distance and, on the geodesic between the markings, vertices that are distance $1$ from the boundaries of these subsurfaces in $\C(S)$.  Hence, the difficulty is determining such a $k$ that is independent of the word $w$.  This is where the cubical geometry developed in Section \ref{quasi} is employed. In particular, the existence of the compact, locally convex subcomplex $C$, coming from quasiconvexity of $H$, controls the regularity with which one encounters filling subwords in the spelling of $w$. This regularity, along with the order preserving map $X^w$, implies that syllables sufficiently spaced out in the spelling of $w$ correspond to filling subsurfaces. Hence, subsurfaces coming from such syllables cannot both be disjoint from a single vertex of our geodesics, thereby forcing up the amount that $w$ displaces the original marking. We now give the details.

\begin{theorem} \label{main}
Let $\mu$ be a marking of $S$ and let $\phi: A(\Gamma) \to$ Mod$(S)$ be an admissible homomorphism into the mapping class group of the surface $S$. Let $H$ be a quasiconvex subgroup of $A(\Gamma)$. Then there is a $\ell>0$ so that if $H$ is $\ell$-short filling, $\phi (H)$ is a convex cocompact subgroup of the mapping class group. In particular, for any $h \in H$
$$d_S (\mu, \phi(h) \mu) \ge |h|_{A(\Gamma)} / 6\ell - 2.$$
\end{theorem}


\begin{proof}
Given quasiconvex $H < A(\Gamma)$, apply Proposition \ref{quasicontrol} to obtain the constant $\ell$ for the condition that $H$ is $\ell$-short filling.  We show that the orbit map from $H$ to $\mathcal{C}(S)$ sending $h$ to $\phi(h)\mu$ is a quasi-isometric embedding. This will prove that $\phi(H) < \Mod(S)$ is convex cocompact. Recall that since $H$ is quasiconvex, it is undistorted in $A(\Gamma)$ and so we may use the word norm on $A(\Gamma)$ in place of a word norm for $H$. Being quasiconvex, $H$ is finitely generated, and therefore the orbit map is coarsely Lipshitz.  In detail, if $h_1, \ldots, h_s$ generates $H$ and $C = \max \{d_S(\mu, \phi(h_i) \mu) : 1\le i \le s\}$ then $d_S(\mu, \phi(h)\mu) \le C \cdot |h|$, where $h \in H$ and $| \cdot |$ denotes the word norm in these generators. Hence, it suffices to prove the lower bound on curve complex distance given in the statement of the theorem.

Given $h \in A(\Gamma)$, choose a normal representative $w \in \mathrm{Min}(h)$.  Then $w = x_1\cdots x_N$ where $x_i$ are standard generators (not syllables) and $N=|h|_{A(\Gamma)}$.  Let $m$ be the smallest integer strictly greater than $N/6\ell-1$.  We can partition $w$ into $m$ subwords of length $6\ell$ by setting
\[W_j = x_{I(j)+1} \cdots x_{I(j+1)} \qquad I(j) = 6\ell\cdot(j-1)\]
so that $w = W_1\cdots W_mW_{m+1}$, where $W_{m+1}$ is some possibly empty suffix.

Each $W_j$ admits a decomposition
\[W_j = P_jQ_jR_j, \quad \text{ where } \quad |P_j| = 3\ell \quad |Q_j|= \ell \quad |R_j| = 2\ell,\]
where $| \cdot|$ denotes the length of each subword; recall that subwords of $w$ are all minimal length for the elements of $A(\Gamma)$ they represent. Let $p_j$ be the first syllable of $P_j$ that is also a syllable of $w$ for $1 \le i \le m$.  There are $m$ of these, and we claim that these syllables can be associated to distinct vertices in a fixed geodesic $[\mu,\phi(h)\mu]$ in the curve complex, as follows.

Let $Y_j \in \subs(w)$ correspond to the syllable $p_j$, that is, $Y_j  = X^w(p_j)$. Here the index corresponds to the $W_j$ subwords, not the syllable indexing of $w$. Since $\phi$ is admissible, we know that $d_{Y_j} (\mu,\phi(h) \mu) \ge K \ge K_{BGI}$.  By Theorem \ref{BGIT}, any curve complex geodesic $[\mu,\phi(h)\mu]$ contains a vertex $v_j$ with empty projection to $Y_j$.  It remains to show that these $v_j$ are distinct for $j$ between $1$ and $m$;  then we can conclude that
\[d_S (\mu, \phi(h) \mu) \geq m-1 \geq |h|_{A(\Gamma)} / 6\ell - 2,\]
completing the proof of the theorem.  Towards contradiction, suppose there exists $j < k$ such that $Y_j$ and $Y_k$ contribute the same vertex in $[\mu,\phi(h)\mu]$.  This means there exists a curve $\gamma = v_j = v_k$ disjoint from both $Y_j$ and $Y_k$.  

Now consider each $Q_j$ and recall $|Q_j| = \ell$.  Lemma \ref{blockexistence} ensures $Q_j$ contains a filling block for $w$.  In particular, this gives a syllable $q \in \syl(w)$ that falls in $Q_j$ and is associated to a subsurface $Z \in \subs(w)$ for which $\pi_Z(\gamma)$ is not empty.

Let us show that $Z$ and $Y_j$ are ordered; by admissibility it suffices to show that $q$ and $p_j$ are ordered.  Recall $p_j$ is the first syllable of $P_j$ that is also a syllable of $w$, and $P_j$ has length $3\ell$.  Lemma \ref{shortsyllables} assures us that the suffix of $p_j$ in $P_j$ has length at least $2\ell$.  Since we know $q$ lies in $Q_j$, we see that the shortest subword of $w$ containing $p_j$ and $q$ has the form $p_jMq$, where $M$ has length at least $2\ell$.  If $p_j$ and $q$ are not ordered, then Lemma \ref{subwords} says $M$ has normal representative $LR$ where $p_j$ commutes $L$ and $q$ commutes with $R$.  On the other hand, since $M$ has length at least $2\ell$, one of the pair $L,R$ is longer than $\ell$, and thus contains a filling block (by Lemma \ref{blockexistence}) which includes syllables ordered with each of $p_j$ and $q$ (by Lemma \ref{blockcommutation}).  This contradicts the characterization of $L$ and $R$ from Lemma \ref{subwords}.  So it must be $p_j$ and $q$ are ordered.  Then so are $Y_j$ and $Z$.  As $w$ is a normal word, evidently $Y_j \prec Z$.

Similarly, because $R_j$ has length at least $2\ell$, we can conclude that $Z \prec Y_k$.  We have shown $Y_j \prec Z \prec Y_k$.  By Lemma \ref{orderedtriple}, since $\gamma$ is disjoint from both $Y_j$ and $Y_k$, it must also be disjoint from $Z$.  This contradicts our choice of $Z$ such that $\pi_Z(\gamma)$ is not empty.  Hence, no curve is simultaneously disjoint from $Y_j$ and $Y_k$ for $j \neq k$.
\end{proof}

Observe that in the proof of Theorem \ref{main} the lower bound on translation distance is proven using only the fact that, in a normal representative of a word in $H$, filling blocks occur with a definite frequency. More precisely, we have shown the following corollary.

\begin{corollary}\label{filling}
For an admissible homomorphism $\phi: A(\Gamma) \to \Mod(S)$, suppose $H \le A(\Gamma)$ satisfies the following: there exists $\ell \ge 0$ such that, for every $h\in H$, any normal form $w\in \mathrm{Min}(h)$ has the property that every subword of $w$ of length $\ge \ell$ contains a filling block. Then $\phi(H) < \Mod(S)$ is convex cocompact.
\end{corollary}

Since pseudo-Anosov mapping classes of $H$ are necessarily filling, we have the following version of Theorem \ref{main}.

\begin{corollary}\label{pAcor}
Let $\phi: A(\Gamma) \to \Mod(S)$ be an admissible homomorphism and $H \le A(\Gamma)$ is quasi-convex. If $\phi(H)$ is purely pseudo-Anosov, then it is convex cocompact in $\Mod(S)$.
\end{corollary}

Observe that although Corollary \ref{pAcor} is simpler to state, the Theorem \ref{main} provides a means of verifying that a quasiconvex subgroup $H \le A(\Gamma)$ is actually purely pseudo-Anosov, assuming that one knows the quasiconvexity constant for $H$.

\section{Necessity of quasiconvexity in $A(\Gamma)$}\label{sec:proofnecessary}
Fix an admissible homomorphism $\phi: A(\Gamma) \to \Mod(S)$. Having shown that for $H < A(\Gamma)$ with $\phi(H)$ purely pseudo-Anosov, quasiconvexity of $H$ in $A(\Gamma)$ implies convex cocompactness of $\phi(H)$ in $\Mod(S)$, it is natural to ask whether quasiconvexity of $H$ is a necessary condition for $\phi(H)$ to be convex cocompact. To answer this question, we first review some facts about the geometry of Teichm\"{u}ller space.

Recall that the $\epsilon$-thick part of Tecihm\"{u}ller space is the region determined by
$$\mathcal{T}_{\ge \epsilon}(S) = \{X \in \mathcal{T}(S): \ell_{X}(\gamma) \ge \epsilon \text{ for all } \gamma \in \mathcal{C}^0(S)  \} $$
and a set in $\mathcal{T}(S)$ is called \emph{$\epsilon$-cobounded} if it resides in the $\epsilon$-thick part.
Although Teichm\"{u}ller space itself is in no ordinary sense negatively curved, a driving principle in the study of the coarse geometry of $\mathcal{T}(S)$ is that the thick part $\mathcal{T}_{\ge \epsilon}$ has many hyperbolic-like features. See for example \cite{KLSurvey}.  One manifestation of this principle is Minsky's theorem that Teichm\"{u}ller geodesics which remain in the thick part of Teichm\"{u}ller space are strongly contracting. More precisely,

\begin{theorem}[Theorem 4.2, \cite{Minsky96}]\label{Minsky}
Given $K$ and $\epsilon$, there exists a $B$ so that if $\alpha$ is a $K$-quasigeodesic path in $\mathcal{T}(S)$ whose endpoints are connected by an $\epsilon$-cobounded Teichm\"{u}ller geodesic $\tau$, then $\alpha$ is contained in the $B$-neighborhood of $\tau$.
\end{theorem}

We can now show that quasiconvexity of $H < A(\Gamma)$ is necessary for $\phi(H)$ to be convex cocompact in $\Mod(S)$. The idea is that quasiconvexity in $\mathcal{T}(S)$ can be pulled back to quasiconvexity in $A(\Gamma)$. Recall that admissible homomorphisms induce quasi-isometric orbit maps into $\mathcal{T}(S)$.  Therefore the next theorem, along with Theorem \ref{construction}, completes the proof of Theorem \ref{intromain}.
\begin{theorem}\label{necessity}
Consider $A(\Gamma) < \text{Mod}(S)$ so that for $X \in \mathcal{T}(S)$ the orbit map
$$\phi: A(\Gamma) \to \mathcal{T}(S) $$
$$g \mapsto g\cdot X$$
is a $K$-quasi-isometric embedding, where $\mathcal{T}(S)$ is given the Teichm\"{u}ller metric.  If $H < A(\Gamma)$ is convex cocompact as a subgroup of Mod($S$), then $H$ is quasiconvex in $A(\Gamma)$.
\end{theorem}

\begin{proof}
Fix $D$, the quasiconvexity constant for $H\cdot X$ in $\mathcal{T}(S)$. Note that $N_D(H\cdot X)$ is $\epsilon$-cobounded for some $\epsilon > 0$. This follows from the facts that $X$ is $\epsilon'$-thick and that, if $Y \in \mathcal{T}(S)$ with $d_{\mathcal{T}(S)}(X,Y) \le D$, then for any simple closed curve $\gamma$, $l_Y(\gamma) \ge e^{-2D} l_X(\gamma) \ge e^{-2D} \epsilon'$.

Let $h, h' \in H$ and suppose $\alpha$ is a geodesic path in $A(\Gamma)$ from $h$ to $h'$. Then $\phi (\alpha)$ is a $K$-quasigeodesic joining $h\cdot X$ and $h'\cdot X$. Let $\tau$ be the Teichmuller geodesic between these points. By quasiconvexity, $\tau \subset N_D(H\cdot X)$ and so $\tau$ is $\epsilon$-cobounded. 

By Theorem \ref{Minsky}, there exists a $B$ depending only on $\epsilon$ and $K$ so that any $K$-quasigeodesic with endpoints on $\tau$ is contained in a $B$-neighborhood of $\tau$. Hence,
$\phi(\alpha) \subset N_B(\tau) \subset N_{D+B}(H\cdot X).$ So if $p \in \alpha \subset A(\Gamma)$, then there is an $h \in H$ with
$$d_{\mathcal{T}(S)}(\phi(p), h\cdot X) \le D+B $$
and so 
\begin{eqnarray*}
d_{A(\Gamma)}(p,h) &\le& K(d_{\mathcal{T}(S)}(\phi(p), h\cdot X)) + K\\
&=& K(D +B) +K
\end{eqnarray*}
This implies that $\alpha \subset N_{K(D+B)+K}(H)$, as required.

\end{proof}

\section{Some examples}\label{sec:examples}
In this section we produce some examples of convex cocompact subgroups of mapping class groups. We encourage the reader to observe the algorithmic nature of these examples, and that in each case quasiconvexity of the subgroup is proven by constructing the compact complex $C$.

\begin{example}
Take $\Gamma$ with $E(\Gamma) = \emptyset$, so that $A(\Gamma) = \mathbb{F}(V(\Gamma))$, the free group on the vertex set of $\Gamma$. In this case, any finitely generated $H \le A(\Gamma)$ is quasiconvex and $C$ is the Stallings core graph for $H$, as explained in \cite{St83}. Briefly, one begins with a rose $R_H$ with rank($H$) petals and subdivides and labels each petal with a generator of $H$ written in terms on the standard generators of $A(\Gamma)$. This produces a map $R_H \to S_{\Gamma} = R_{|V(\Gamma)|}$ and iteratively folding the graph $R_H$ produces a graph $G_H$ and an immersion $G_H \to R_{|V(\Gamma)|}$ with image $H$ at the fundamental group level. This is the desired $C$, since graph immersions are local isometries (when each edge is given unit length). Hence, when $A(\Gamma)$ is free,  for any finitely generated $H < A(\Gamma)$ and any admissible homomorphism $\phi: A(\Gamma) \to \Mod(S)$, $\phi(H)$ is convex cocompact if and only if $\phi(H)$ is purely pseudo-Anosov.
\end{example} 

\begin{example}
Take $\Gamma$ as in Figure $1$ with any admissible realization of $A(\Gamma)$ in Mod$(S)$ fully supported on $X_a,X_b,X_c$. Then $A(\Gamma) = \langle a,b,c| [b,c] = 1 \rangle$ and consider $H = \langle bca, babc \rangle$. We construct the complex $C$ described below. This will be the desired complex once we verify that the induced map $C \to S_{\Gamma}$ is a local isometry, as in Section \ref{cubegeo}.

To construct $C$, we begin as in the example above by building an immersion from a based graph into the $1$-skeleton of $S_{\Gamma}$, see Figure \ref{fig2}. In this case, a single fold is required. This map, however, is far from a local isometry into $S_{\Gamma}$ because of the absence of $2$-cells. We add such $2$-cells to the graph representing the commuting relations between $b$ and $c$, first adding the squares between adjacent occurrences of $b$ and $c$ in the graph. The resulting complex admits an obvious cubical map into $S_{\Gamma}$. This, however, is not a local isometry as the image of the link at the base point is not a full subcomplex of lk$(v, S_{\Gamma})$. To rectify this, attach an additional square representing commuting relations between $b$ and $c$, giving four $2$-cells in total.  Call the resulting complex $C$. We then verify that the induced map on the link of each vertex satisfies the conditions for $C \to S_{\Gamma}$ to be a local isometry. Hence, $H$ is a quasiconvex subgroup of $A(\Gamma)$. From this, we can readily check that each $h \in H$ is filling and so by Theorem \ref{construction}, $H$ is convex cocompact as a subgroup of the mapping class group. 
\end{example}

\begin{figure}[htbp]
\begin{center}
\includegraphics[height=60mm]{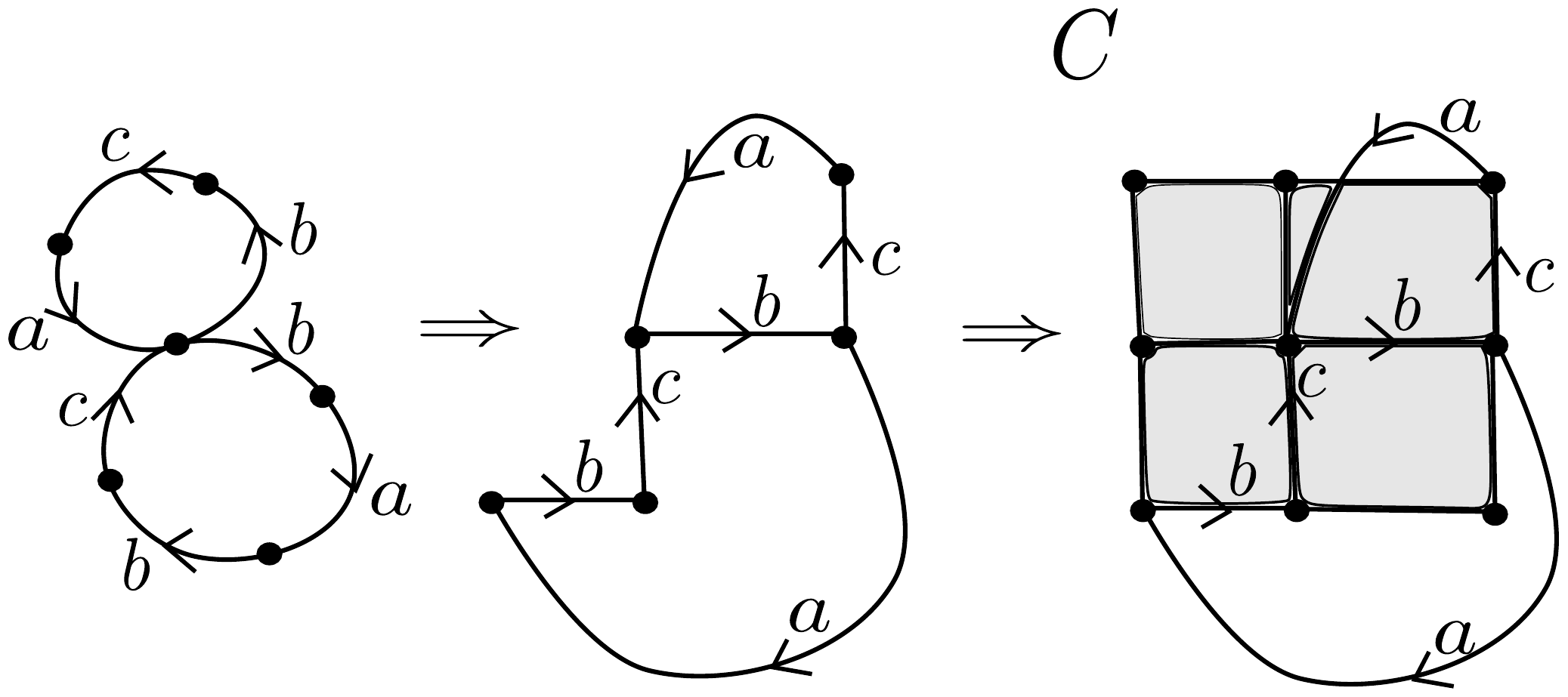}
\caption{Constructing $C$}
\label{fig2}
\end{center}
\end{figure}

\begin{example}
We augment the example above by adding the element $b^2c^2a^2$ to the generating set of $H$. Note that inspection of the complex $C$ reveals that this word is not already in the subgroup $H$. Call the resulting subgroup $H' = \langle abc,cba, a^2b^2c^2 \rangle$.  

To build $C'$, the complex with a local isometry to $S_{\Gamma}$ with fundamental group $H'$, we add to the complex $C$ with a relative version of the construction above. This requires attaching a loop labeled $b^2c^2a^2$ to the base vertex, performing $2$ folds, and adding $5$ squares. Once this is done, it is again easily seen that $C'$ has the required properties. Hence, $H'$ is a convex cocompact subgroup of the mapping class group (through the homomorphism fixed above).
\end{example}

In the above examples, we used the existence of the complex $C$ to verify that our subgroup was quasiconvex with all filling elements. The following example shows how attempting to construct $C$ can also show that a particular group will not be purely pseudo-Anosov. 

\begin{example}
As in the previous example, we add to the subgroup $H$ without changing $A(\Gamma) \le \Mod(S)$. Suppose we wish to add the generator $bca^2$ to obtain the subgroup $H'' = \langle abc,cab,a^2bc \rangle$. Attaching the necessary cells to the complex $C$, we see that some word in $H''$ is conjugate to $a$ and hence not filling. Adding the word $b^{-1}ca^2$, however, will result in a quasiconvex subgroup all of whose elements are filling. Again, this is immediate by building the associated complex.
\end{example}

\section{Short-translation convex cocompact subgroups of $\Mod(S)$}\label{shortsec}

As an application of our main theorem, we construct convex cocompact subgroups of mapping class groups $\Mod(S_g)$ whose orbit maps into the curve complex have Lipschitz constant on the order of $1/g$ , where $g$ denotes the genus of the surface. This is in contrast to other examples of convex cocompact subgroups occurring in the literature, where one does not have such control over the group constructed. This is clear in the first two constructions of convex cocompact subgroups of mapping class groups that are discussed in the introduction. It is also the case that, in the mapping class group of a once-punctured surface, a pseudo-Anosov in the image of the point-pushing map (the map $\pi_1(S_g) {\to} \Mod(\mathring{S}_g)$ discussed in Section \ref{intro}) has minimum possible translation length equal to $1$, independent of genus \cite{Zhang}.

\begin{figure}[htbp]
\begin{center}
\includegraphics[height=80mm]{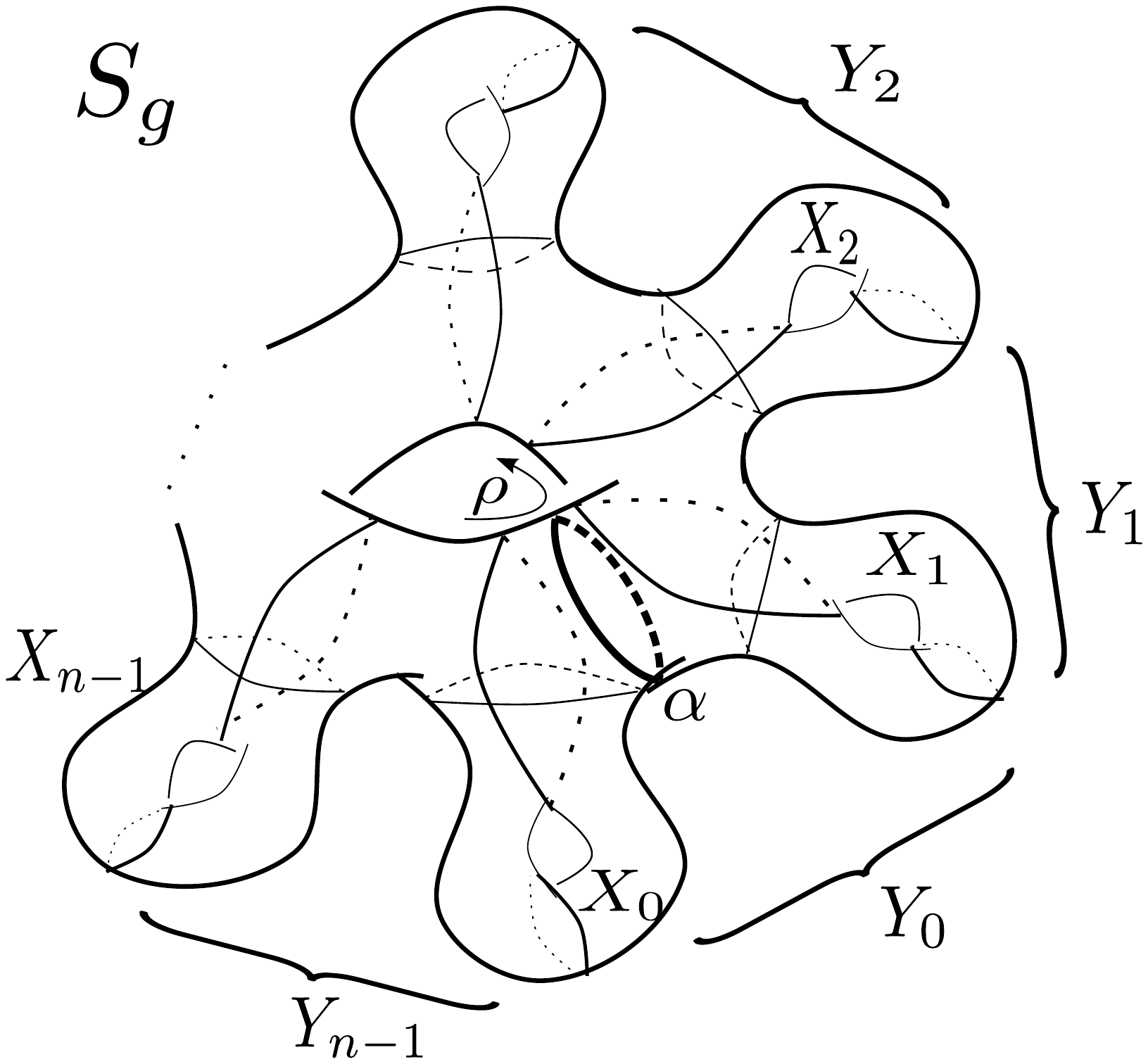}
\caption{$S_g$ with subsurfaces}
\label{fig3}
\end{center}
\end{figure}

Fix $n\ge 2$ and let $S_g$ be a surface of genus $g =n+1$. Referring to Figure \ref{fig3}, let $f_0$ be a mapping class fully supported on $X_0$, let $g_0$ be fully supported on $Y_0$, and let $\rho$ be the order-$n$ counterclockwise rotation shown (see Section \ref{mcgbasics} for the definition of fully supported). Here $X_0$ is a once-punctured torus and $Y_0$ is a four-times punctured sphere. We will require these maps to be pseudo-Anosov with sufficiently large translation length in the curve complex. Define
\[ f_i = \rho^i f_0 \rho^{-i} \qquad \text{ and } \qquad g_i = \rho^i g_0 \rho^{-i}. \]
By construction, $f_i$ is fully supported on $X_i$ and $g_i$ on $Y_i$, where $X_i = \rho^i X_0$ and $Y_i = \rho^i Y_0$.  Note that we have only defined finitely many distinct maps, and may consider the index $i$ as an integer mod $n$.

Observe by disjointness of the subsurfaces involved, we have the following:
\begin{itemize}
\item $f_i$ commutes with $f_j$ for all $1 \le i,j \le n$,
\item $g_i$ commutes with $g_j$ for all $1\le i,j \le n$, and
\item $f_i$ commutes with $g_j$ for $j \neq i-1, i$.
\end{itemize}

Hence, if $\Gamma$ is the coincidence graph for these subsurfaces, $\Gamma^c$ is a cycle of length $2n$, alternating $f_i$ and $g_i$ in cyclic order. If the initial $f_0, g_0$ are chosen with large enough translation length in the curve complexes of $X_0$ and $Y_0$ respectively, where ``large enough'' is independent of $n$, then by Theorem \ref{CLMtheorem} we have an admissible embedding
$$A(\Gamma) \to \Mod(S),$$
and so we identify $A(\Gamma)$ with its image under this homomorphism.

We now demonstrate a few calculations.
Set $\phi = \rho g_0 f_0$ and $\psi = \phi^n$, so
\begin{eqnarray*}
\psi = \phi^n &=& \rho g_0 f_0 \rho g_0 f_0 \ldots \rho g_0 f_0 \\
&=&  \rho g_0 f_0 \rho^{-1} \rho^2 g_0 f_0 \rho^{-2} \ldots \rho^{n} g_0 f_0 \\
&=& g_1f_1g_2f_2 \ldots g_nf_n \\
&=&( g_1g_2 \ldots g_{n-1})(f_1g_n)(f_2f_3 \ldots f_n).
\end{eqnarray*}
Let $w_i = (\rho g_0^if_0^i)^n$, so that $w_1 = \psi$ and $w_i$ replaces each $g_j$ and $f_j$ in $\psi$ with its $i$th power:
\[w_i =  ( g_1^ig_2^i \ldots g_{n-1}^i)(f_1^ig_n^i)(f_2^if_3^i \ldots f_n^i)\]
Let $\alpha$ be the separating curve in $Y_0$ disjoint from both $X_0$ and $X_1$, as shown in Figure \ref{fig3}. 

We can now prove the following theorem, which gives examples of convex cocompact subgroups of the mapping class group whose generators have small translation length in the curve complex.

\begin{theorem}
Let $S$ be as above with genus $g = n+1$ for $n \ge2$, and let $H = \langle w_1, \ldots w_N \rangle < A(\Gamma) < \Mod(S_g)$ as above, for any $N\ge1$. Then $H$ is a convex cocompact subgroup of $\Mod(S)$ and with $\alpha$ as above,
$$d_{\mathcal{C}(S)}(\alpha,w_i^p\alpha) \le 2,$$
for all $i = 1 ,\ldots, N$ and $p \leq \frac{n}{2}$. In particular, $\ell_S(w_i) \le \frac{4}{g-1}.$
\end{theorem}

\begin{proof}The statement about translation length follows from Theorem \ref{Lipschitz} below.

Proving convex cocompactness of $H$ require additional notation as well as a few lemmas. Set for $1 \le i \le N$
$$B_i = g_1^i\ldots g_{n-1}^i,$$
$$M_i= f_1^i g_n^i,$$
$$E_i = f_2^i\ldots f_n^i.$$
Observe that $B_i^{-1} = B_{-i}$, $E_i^{-1} = E_{-i}$, but $M_i^{-1} \neq M_{-i}$. In what follows we do not allow negative indices for the $M_i$. Symbols without subscripts will be used when the subscript is not important (e.g. $B=B^{-1}$ but $M \neq M^{-1}$). In any case, the occurrence of such a symbol will aways denote a non-zero subscript (and power). The first lemma gives a procedure for assembling a \emph{particular} normal form for any $h \in H$. 

\begin{lemma}
With notation as above, we have the following facts:
\begin{enumerate}
\item $w_i = B_iM_iE_i$
\item $w_iw_j^{-1} = B_iM_iE_{i-j}M_j^{-1}B_j^{-1}$ \\
	$w_i^{-1}w_j = E_i^{-1}M_i^{-1}B_{j-i}M_jE_j$
\item $B_iM_j^{\pm 1}E_k$ and its inverse are filling words where $j\ge 1$.
\item Suppose that $h$ is a word of $A(\Gamma)$ that is a concatenation of $B$s, $M$s, and $E$s so that adjacent pairs appear as $BM$, $ME$, $EB$ or their inverses ($M^{-1}B$, $EM^{-1}$, or $BE$). Then $h$ is in normal form with respect to the standard generators $f_1,\ldots,f_n,g_1,\ldots, g_n$ of $A(\Gamma)$.
\end{enumerate}
\end{lemma}

\begin{proof} 
Fact $1$ is immediate from the definitions, as is Fact $2$ once it is recalled that all letters used in the spelling of $B$ (and $E$) pairwise commute. Fact $3$ follows since these words have the least number of syllables in their conjugacy classes and the supports of the generators used in their spelling fill $S$. Fact $4$ requires some verification. First observe that any word with at most one occurrence of $B,M$ and $E$ is in normal form, since in this case every generator occurs in a syllable at most once. Suppose toward contradiction that $h$ as in Fact $4$ is not in normal form. Then there is a syllable that can be shuffled to the right using move $(3)$ and combined with another syllable elsewhere in the word. This syllable must commute with every letter in between. By Fact $3$ and the fact that each generator occurs only within a $B,M$ or $E$ such shuffling can only occur between consecutive occurrence of the same symbol $B,M$ or $E$;  otherwise it would have to commute with a filling word, an impossibility. Hence, the remaining cases are when the syllables that are to be combined occur between consecutive $B,M$ or $E$ expressions. For the $B$ case, such a syllable is a power of $g_j$ and must be commuted passed an entire expression $E$ to combine with another power of $g_j$. This, however, is a contradiction since $g_j$ fails to commute with both $f_j$ and $f_{j+1}$, one of which occurs in $E$. The case where the syllable to be combined is contained in $E$ is similar. The final case is when the syllable is contained in $M$ or $M^{-1}$ and is therefore either a power of $f_1$ or $g_n$. However, in $M$, $f_1$ cannot be commuted past $g_n$, and $g_n$ cannot be commuted past $f_n$ which occurs in the expression $E$. In either case, neither syllable can be shuffled right to combine with another syllable. In $M^{-1}$, $g_n^{-1}$ cannot commute past $f_1^{-1}$, and $f_1^{-1}$ cannot commute past $g_1$ which occurs in the expression $B$, so we have a similar contradiction. This completes the proof that no syllable of $h$ can be combined with another syllable and shows that $h$ is in normal form.
\end{proof}

By Corollary \ref{filling}, to prove that $H$ is a convex cocompact subgroup of $\Mod(S_g)$ it suffices to show that there is an $\ell \ge 0$ so that for $h \in H$, any $w \in \mathrm{Min}(h)$ has the property that any subword of $w$ of $A(\Gamma)$-length greater than $\ell$ contains a filling block.

First suppose that $h \in H$ is written in terms of the generators $w_1, \ldots, w_N$ with no consecutive occurrences of a generator and its inverse (i.e. $h$ is written as a reduced word in the free group generated by the symbols $w_1, \ldots, w_N$). Now use Facts $1$ and $2$ above to transform $h$ into it its $BME$ spelling. We note that since $h$ is reduced in the free group generated by the $w_i$, within the $BME$ spelling no subscript is zero (see Fact $2$) and so expressions $E$ and $B$ are always separated by a $M$ or $M^{-1}$. To verify that this rewriting of $h$ is in normal form, we use Fact $4$ above. This is immediate from Fact $2$ and the fact that no zero subscript occurs in the rewriting. Denote the normal form representative of $h$ so obtained by $w^*$ and refer to it as the \emph{$BME$ normal form}. The next lemma shows that this normal form has the desired property.

\begin{lemma}\label{existsb}
For any $h\in H$, its $BME$ normal form $w^*$ has the property that every subword of $w^*$ with $A(\Gamma)$-length at least $b  = 3Nn +4N$ has a filling block.
\end{lemma}

\begin{proof}
Assume that $r$ is a subword of $w^*$ of $A(\Gamma)$-length at least $3Nn+4N$. In this spelling, no more than the first $Nn$ letters occur in $r$ before the first full occurrence of $B, M$ or $E$. Whichever symbol occurs first, we see that there is a terminal subword of $r$ of length at least $2Nn+4N$ that begins with either $B,M$ or $E$. Checking the possible cases, we see that within $r$ we have a subword $w'$ which is a consecutive occurrence of $BME$ up to cyclic permutations and inverses.  By Fact $3$ of the previous lemma, $w'$ is a filling block. Thus, $r$ contains a filling block as required.
\end{proof}

Our final lemma needed for the proof gives a strong ordering property for the $BME$ normal forms. Note that our choice of $A(\Gamma) \le \Mod(S_g)$ is such that the support of every generator is required to fill the surface $S_g$.

\begin{lemma} \label{orderafterL}
There is an $L \ge0$ so that if  $h\in H$ has $BME$ normal form $w^*$ and $s$ is a syllable of $w^*$, then any syllable separated from $s$ in $w^*$ by $L$ or more syllables is ordered with $s$. Hence, no more than $2L$ syllables in $h$ are unordered with $s \in \syl(h)$ in any normal form for $h$.
\end{lemma}

\begin{proof}
Let $\Gamma^c$ denote the complement graph of $\Gamma$. As previously observed, $\Gamma^c$ is connected; let $d$ denote its diameter. Take $L = d \cdot b$.

Now let $t$ be a syllable of $w^*$ so that at least $L$ syllables occur between $s$ and $t$ in $w^*$. Let $v$ and $u$ be standard generators of $A(\Gamma)$ so that $s$ is a power of $v$ and $t$ is a power of $u$. Then, for some $k \le d$, there is a path $v =v_0,v_1,\ldots,v_k =u$ in $\Gamma^c$; by construction, $v_i$ and $v_{i+1}$ do not commute. Since $k \cdot b \le d\cdot b$, we may select $w_1, w_2, \ldots, w_{k-1}$ a sequence of disjoint filling blocks between $s$ and $t$ occurring in the order as read from $s$ to $t$. Being filling blocks, each $w_i$ must include every generator, so for $1\le i< k$ we choose $s_i$ to be a syllable in $w_i$ that is a power of $v_i$. Then by construction
$$s \prec s_1 \prec \ldots \prec s_{k-1} \prec t. $$
Hence, by transitivity $s \prec t$.  For the conclusion on the lemma, recall that the number of syllables unordered with a given $s$ in $\syl(h)$ is independent of choice of normal form.
\end{proof}

We can now conclude the proof of the theorem. Set $\ell' = b  + 4LN + 1$ and let $w$ be any normal form for $h\in H$. Towards contradiction, suppose that there is a subword $r$ of $w$ of length at least $\ell'$ that does not contain an occurrence of some generator, call this generator $v$. Enlarge this subword, continuing to denote it $r$, so that either the first or last syllable of $r$ is a power of $v$ and the other terminal syllable is either a terminal syllable of $w$ or another syllable that is a power of $v$. This is possible since $w$ itself is filling and so must contain at least one occurrence of a power of $v$. 

Let $w^*$ be the $BME$ normal form for $h$; note that, by Lemma \ref{existsb}, there is no subword of $w^*$ longer than $b$ that fails to contain each generator of $A(\Gamma)$. Since $w,w^* \in \mathrm{Min}(h)$, these words are related by a repeated application of the type $(3)$ moves described in Section \ref{normalforms}. There are two cases, depending on whether $r$ contains one or two occurrences of $v$. If the first and last syllables of $r$ are powers of $v$ then denote these by $s_1$ and $s_2$. Since between these syllables in $w$ there is no occurrence of $v$, the type $(3)$ moves from $w$ to $w^*$ must commute $\ell' -b$ letters past either $s_1$ or $s_2$. Hence, at least $(\ell' -b)/2$ letters must be commuted past, say, $s_1$.  Since no syllable contains more than $N$ letters, at least $(\ell' -b)/(2N)$ syllables are commuted past $s_1$ in the process of applying type $(3)$ moves. All of these syllables must be unordered with $s_1$. However, by Lemma \ref{orderafterL} the number of syllables unordered with $s_1$ is bounded by $2L$. Hence,
$$\frac{\ell' -b}{2N} \le 2L, $$
contradicting our choice of $\ell'$. The case where $r$ contains only one occurrence of the generator $v$ and, hence, contains the first or last syllable of $w$ is similar and left to the reader. 

Let $\ell = \ell' + 2N$.  Given $w'$ a subword of $w$ of length at least $\ell$, let $r$ be the subword consisting of the middle $\ell'$ letters of $w'$.  We have seen that $r$ contains an occurrence of each generator.  By Lemma \ref{subsfills}, $r$ is a filling block for $w$ unless one or both of its terminal syllables are not complete syllables in $w$.  Appending up to $N$ letters to the left and the right of $r$ gives filling block for $w$ contained in $w'$.
\end{proof}

Let $S_g$ be as above with genus $g = n+1$ for $n \ge2$, and let $H = \langle w_1, \ldots w_N \rangle \le A(\Gamma) \le \Mod(S_g)$ as above, for any $N\ge1$.  For any element $h \in H$, let $|h|_H$ denote its word length with respect to the generating set $\{w_1, \ldots w_N\}$.  Let $\alpha$ be the curve in Figure \ref{fig3}.

\begin{theorem}\label{Lipschitz}For any $h \in H$,
\[d_{\mathcal{C}(S)}(\alpha,h\alpha) \leq |h|_H\cdot\frac{4}{g-1} + 2\]\end{theorem}

\begin{proof} It suffices to prove the following lemma:
\begin{lemma}\label{d2}If $|h|_H \leq n/2$ then $d_{\mathcal{C}(S)}(\alpha,h\alpha)\leq 2$.\end{lemma}

This is because, where $m$ is the largest integer less than $|h|_H \cdot (2/n)+1$, we may write $h = h_1\cdots h_m$ for some $h_i$ such that $|h_i|_H \leq \frac{n}{2}$. Using the triangle inequality, we have
\begin{eqnarray*}
d_{\mathcal{C}(S)}(\alpha,h\alpha) &\leq& d_{\mathcal{C}(S)}(\alpha,h_1\alpha) + d_{\mathcal{C}(S)}(h_1\alpha,h_1h_2\alpha) + \dots + d_{\mathcal{C}(S)}(h_1h_2\cdots h_{m-1}\alpha,h\alpha)\\
&=& \sum_{i=1}^m d_{\mathcal{C}(S)}(\alpha,h_i\alpha) \leq 2m \leq |h|_H \cdot (4/n)+2.
\end{eqnarray*}
To obtain the theorem, recall $g=n+1$.

So we are tasked with proving the lemma above.  We rely on an obvious principle: if $S$ is connected has a proper subsurface $S'$ containing the two curves $\alpha$ and $\beta$, then $d_{\mathcal{C}(S)}(\alpha,\beta) \leq 2$, because $\alpha$ and $\beta$ are each disjoint from the curves of $\del S'$.

For a list of subsurfaces $Z_1,\dots,Z_k$ of $S$, we will write $\Fill\{Z_1,\dots,Z_k\}$ for the essential subsurface of $S$ of minimal complexity that contains each of $Z_1,\dots,Z_k$ (up to isotopy). In other words, $\Fill\{Z_1,\dots,Z_k\}$ is the subsurface $Z$ of $S$ that contains each subsurface in the collection and has the property that any curve that projects nontrivially to $Z$ also has nontrivial projection to some $Z_i$.  If each $Z_i$ in the collection equals either some $X_{s}$ or $Y_{t}$, observe that $\Fill\{Z_1,\dots,Z_k\}$ is a proper subsurface of $S$ unless \emph{all} of the subsurfaces $X_1, \ldots, X_n$ and $Y_1, \ldots, Y_n$ occur in the collection. Hence, the set of all $X_{s}$ and $Y_{t}$ fill $S$, but no proper subset of this collection of subsurfaces fills $S$.

\begin{figure}[htbp]
\begin{center}
\includegraphics[height=70mm]{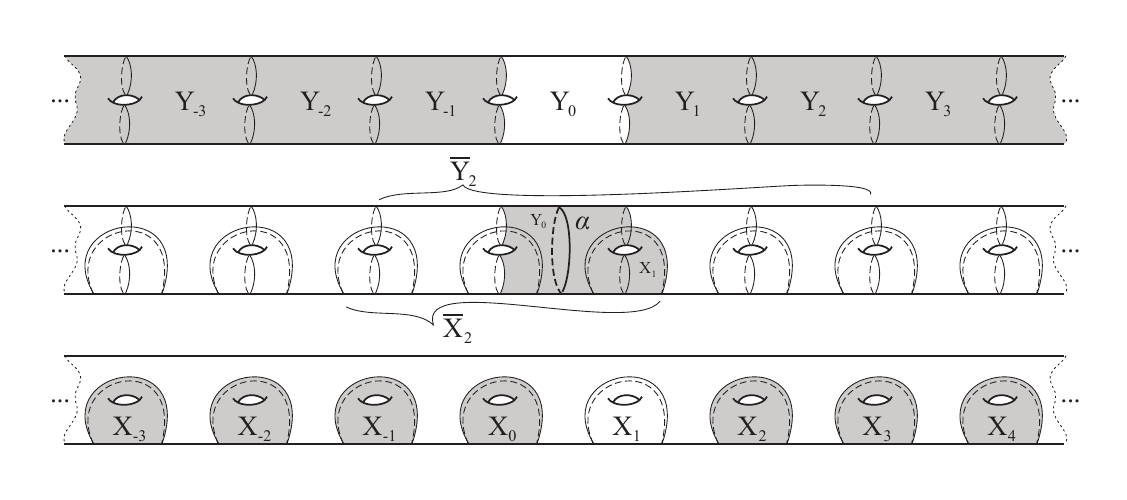}
\caption{Three views of $S_g$.  The supports of $B$, $M$, and $E$ are shaded in the top, middle, and bottom copies respectively.}
\label{BME}
\end{center}
\end{figure}

Below, we write $w$ to represent any $w_i$ and $v$ to represent any $w_i^{-1}$.  That is, any appearance of $w$ or $v$ satisfies $w=BME$ or $v=EM^{-1}B$ for some suitabe $B,M,$ and $E$.  It turns out $H$ is a free group, although this fact is not necessary in order to obtain upper bounds on displacement.  We think of a given $h$ as a product of generators of type $w$ or $v$.  It is tedious but straightforward to verify by inspection the following statements:
\begin{itemize}
\item $\alpha, v\alpha \subset \Fill\{X_0,Y_0\}$ which is a proper subsurface of $S$ (written ``$\subsetneq S$'') if $n\geq 2$
\item $\alpha, w\alpha \subset \Fill\{Y_0,X_1,Y_1\}\subsetneq S$ if $n\geq 2$
\item $\alpha, wv\alpha \subset \Fill\{Y_1,X_0,Y_0,X_1,Y_1\}\subsetneq S$ if $n\geq 3$
\item $\alpha, vw\alpha \subset \Fill\{X_0,Y_0,X_1,Y_1,X_2\}\subsetneq S$ if $n\geq 3$
\item $\alpha, w^2\alpha \subset \Fill\{Y_{-1},X_0,Y_0,X_1,Y_1,X_2,Y_2\}\subsetneq S$ if $n\geq 4$
\item $\alpha, v^2\alpha \subset \Fill\{X_{-1},Y_{-1},X_0,Y_0,X_1\}\subsetneq S$ if $n\geq 5$
\end{itemize}
These statements are sufficient to verify Lemma \ref{d2} if $n \leq 5$.  They also set forth enough initial cases to complete the proof of Lemma \ref{d2} via induction.

For $k \geq 2$, define the subsurfaces:
\begin{eqnarray*}
\overline{X}_k &=& \Fill\{X_i\thinspace,\thinspace Y_j:-k<i<k \text{ and } -k<j<k-1\}\\
\overline{Y}_k &=& \Fill\{Y_i\thinspace,\thinspace X_j:-k<i\leq k \text{ and } -k+1<j\leq k\}
\end{eqnarray*}
Using our comment above about how many subsurfaces $X_i$ and $Y_j$ are required to fill $S$, one can check that, for $k \leq n/2$, $\overline{X}_k$ and $\overline{Y}_k$ are proper subsurfaces of $S$.  Thus we are done if we can establish the claim:
\begin{itemize}
\item[($\star$)] If $|h|\leq k$ then $h\alpha$ is contained in either $\overline{X}_k$ or $\overline{Y}_k$. 
\end{itemize}
The claim for $k=2$ is an outcome of the definitions and the bulleted statements above.  So it remains to assume the truth of ($\star$) for $k\geq 2$, and prove ($\star$) for $k+1$.

Given $h$ such that $|h|=k+1$, either $h = wh'$ or $h = vh'$ where $|h'| = k$ satisfies ($\star$).  We also know $h'\alpha$ is contained in either $\overline{X}_k$ or $\overline{Y}_k$.  We analyze all possibilities below.  Recall that $w=BME$ and $v=EM^{-1}B$ where $B$ is supported entirely on $\cup Y_i$ and $E$ is supported entirely on $\cup X_i$.  Since $k \geq 2$ we are also assured that $M$ is supported within $\overline{X}_k$ as well as within $\overline{Y}_k$.  For the next set of statements---apparent by inspection---observe that, depending on the case, applying $B$ or $E$ may increase the support required to contain the image curve, while $M$ has no effect.  See Figure \ref{BME}.

If $h'\alpha \subset \overline{X}_k = \Fill\{X_{-k+1},Y_{-k+1},\cdots,Y_{k-2},X_{k-1}\}$,
\begin{eqnarray*}
Bh'\alpha &\subset& \Fill\{Y_{-k},X_{-k+1},Y_{-k+1},\cdots,Y_{k-2},X_{k-1},Y_{k-1}\}\\
M^{-1}Bh'\alpha &\subset& \Fill\{Y_{-k},X_{-k+1},Y_{-k+1},\cdots,Y_{k-2},X_{k-1},Y_{k-1}\}\\
vh'=EM^{-1}Bh'\alpha &\subset& \Fill\{X_{-k},Y_{-k},X_{-k+1},Y_{-k+1},\cdots,Y_{k-2},X_{k-1},Y_{k-1},X_k\}=\overline{X}_{k+1}\\
Eh'\alpha &\subset& \overline{X}_k\\
MEh'\alpha &\subset& \overline{X}_k\\
wh'=BMEh'\alpha &\subset& \Fill\{Y_{-k},X_{-k+1},Y_{-k+1},\cdots,Y_{k-2},X_{k-1},Y_{k-1}\} \subset \overline{X}_{k+1}.
\end{eqnarray*}

If $h'\alpha \subset \overline{Y}_k = \Fill\{Y_{-k+1},X_{-k+2},\cdots,X_k,Y_k\}$,
\begin{eqnarray*}
Eh'\alpha &\subset& \Fill\{X_{-k+1},Y_{-k+1},X_{-k+2},\cdots,X_k,Y_k,X_{k+1}\}\\
MEh'\alpha &\subset& \Fill\{X_{-k+1},Y_{-k+1},X_{-k+2},\cdots,X_k,Y_k,X_{k+1}\}\\
wh'=BMEh'\alpha &\subset& \Fill\{Y_{-k},X_{-k+1},Y_{-k+1},X_{-k+2},\cdots,X_k,Y_k,X_{k+1},Y_{k+1}\} = \overline{Y}_{k+1}\\
Bh'\alpha &\subset& \overline{Y}_k\\
M^{-1}Bh'\alpha &\subset& \overline{Y}_k\\
vh'=EM^{-1}Bh'\alpha &\subset& \Fill\{X_{-k+1},Y_{-k+1},X_{-k+2},\cdots,X_k,Y_k,X_{k+1}\} \subset \overline{Y}_{k+1}.\\
\end{eqnarray*}

\end{proof}

\bibliography{biblio}
\bibliographystyle{amsalpha}
\end{document}